\documentclass[pdflatex,sn-mathphys-num]{sn-jnl}

\usepackage{graphicx}%
\usepackage{multirow}%
\usepackage{amsmath,amssymb,amsfonts}%
\usepackage{amsthm}%
\usepackage{mathrsfs}%
\usepackage[title]{appendix}%
\usepackage{xcolor}%
\usepackage{textcomp}%
\usepackage{manyfoot}%
\usepackage{booktabs}%
\usepackage{algorithm}%
\usepackage{algorithmicx}%
\usepackage{algpseudocode}%
\usepackage{listings}%

\theoremstyle{thmstyleone}%
\newtheorem{theorem}{Theorem}
\newtheorem{corollary}{Corollary}
\newtheorem{lemma}{Lemma}

\newtheorem{proposition}[theorem]{Proposition}%

\theoremstyle{thmstyletwo}%
\newtheorem{example}{Example}%
\newtheorem{remark}{Remark}%

\theoremstyle{thmstylethree}%
\newtheorem{definition}{Definition}%

\begin{document}

\title[{\bf\large Existence of Solutions and Relative Regularity Conditions for Polynomial Vector Optimization Problems}]{{\bf\large Existence of Solutions and Relative Regularity Conditions for Polynomial Vector Optimization Problems}}


\author*[1]{\fnm{Danyang} \sur{Liu}}\email{dyliu@cwnu.edu.cn}

\affil*[1]{\orgdiv{School of Mathematics and Information}, \orgname{China West Normal University}, \orgaddress{ \city{Nanchong}, \postcode{637009}, \state{Sichuan}, \country{P. R. China}}}


\abstract{In this paper, we establish the existence of the efficient solutions for polynomial vector optimization problems on a nonempty closed constraint set without any convexity and compactness assumptions. We first introduce the relative regularity conditions for vector optimization problems whose objective functions are a vector polynomial and investigate their  properties and characterizations. Moreover, we establish relationships between the relative regularity conditions, Palais-Smale condition, weak Palais-Smale condition, M-tameness and properness with respect to some index set. Under the relative regularity and non-regularity conditions, we establish nonemptiness of the efficient solution sets of the polynomial vector optimization problems respectively.  As a by-product, we infer Frank-Wolfe type theorems for a non-convex polynomial vector optimization problem.  Finally, we study the local properties and  genericity characteristics of the relative regularity conditions.}

\keywords{Polynomial vector optimization problem, relative regularity conditions, existence of solutions, genericity}


\pacs[MSC Classification (2020)]{90C29, 90C23, 49K40}

\maketitle

\section{Introduction}\label{sec1}
Throughout,  $\mathbf{R}^{n}$ denotes the $n$-dimensional Euclidean space with  the norm $\|\cdot\|$ and the inner product $\langle \cdot,\cdot\rangle$,  and $\mathbf{R}^{n}_{+}=\{x=(x_1,\cdots,x_n)\in \mathbf{R}^{n}:x_i \ge 0, i=1, \cdots, n\}$. In this paper, we consider the following polynomial vector optimization problem on  $K$:
$$ {\rm{PVOP}}(K, f): \qquad\mbox{Min}_{x\in K}f(x),$$
where $f=(f_{1},  \dots, f_{q}): \mathbf{R}^{n} \mapsto \mathbf{R}^{q}$ is  a vector polynomial such that each component function  $f_{i}$ is a polynomial with its degree $ \mbox{ deg } f_i=d_{i}$, and $K\subseteq \mathbf{R}^{n}$ is a nonempty unbounded closed set (not necessarily convex set or semi-algebraic set  \cite{HHV,BR,Bochnak}). In what follows, we always assume the each component polynomial $f_i$ of the objective function $f$ has a degree $d_{i}\geq 1$.

 Recall that a point $x^{*}\in K$ is said to be a Pareto efficient solution of ${\rm{PVOP}}(K, f)$ if for all $x\in K$, $$f(x)-f(x^{*})\notin - \mathbf{R}^{q}_{+}\backslash \{0\},$$
and   $x^{*}\in K$ is said to be a weak Pareto efficient solution of PVOP$(K, f)$ if for all $x\in K$, $$f(x)-f(x^{*})\notin - \mbox{ int }\mathbf{R}^{q}_{+}.$$
The Pareto efficient solution set and the weak Pareto efficient solution set of ${\rm{PVOP}}(K, f)$ are denoted by $SOL^{s}(K, f)$ and $SOL^{w}(K, f)$ respectively. Obviously, $SOL^{s}(K, f)\subseteq SOL^{w}(K, f)$.  When $q=1$,  ${\rm{PVOP}}(K, f)$ collapses to a polynomial scalar optimization problem denoted by $\mbox{PSOP}(K, f)$, whose solution set is denoted by $SOL(K, f)$.

Existence of efficient solutions play an important role in vector optimization theory. Numerous papers have considered the existence of solutions for the vector optimization problems, see \cite{Borwe,Deng,Kim,DTN,BT1,BT2,24LDY3}. Regularity condition has been  used  in \cite{Lee915} to investigate the existence of solutions and the continuity of the solution mapping for a quadratic programming problem.  Hieu \cite{HV1} established a Frank-Wolfe type theorem for a polynomial scalar optimization problem on a nonempty closed set when the objective function is bounded from below on the constraint set and the regularity condition holds and an Eaves type theorem for non-regular pseudoconvex optimization problems.  Hieu et al. \cite{HV2} proved that the solution set of an optimization problem corresponding to a polynomial complementarity problem is nonempty and compact by using the regularity condition of the polynomial complementarity problem. Meanwhile, some authors investigated the existence of efficient solutions of polynomial vector optimization problems. Kim et al.  \cite{DTN} obtained the nonempty of Pareto efficient solution sets for an unconstrained polynomial vector optimization problem when the Palais-Smale-type conditions hold and the image of the objective vector function has a bounded section. Duan et al. \cite{LGJ1} extended the work of \cite{DTN}. When the Palais-Smale-type conditions hold and the image of the objective vector function has a bounded section, they proved the existence of Pareto solutions of an constrained polynomial vector optimization problem under the regularity at infinity of the constraint set. When $K$ is a convex semi-algebraic set and  $f$ is a convex vector-valued polynomial, Lee et al. \cite{LGJ} proved that  ${\rm{PVOP}}(K, f)$ has a Pareto efficient solution if and only if  the image $f(K)$ of $f$ has a nonempty  bounded section. Recently, by using some powerful tools of asymptotic analysis, Liu et al.\cite{LDY1} studied the solvability for a class of regular polynomial vector optimization problem on a closed constraint set without convexity and semi-algebraic assumptions. Under the weak section-boundedness, convenience and non-degeneracy conditions, Liu et al.\cite{LDY2} obtained Frank-Wofle type theorems for polynomial vector optimization problem by using ways of semi-algebraic geometry. Based on asymptotic notions, Flores-Baz\'{a}n et al.\cite{Flores24} established coercivity properties, coercive and noncoercive existence results for weak efficient solutions of vector optimization problems. Inspired by the above works, in this paper, we study  the existence of Pareto efficient solution of ${\rm{PVOP}}(K, f)$ on a closed constraint set  without convexity and semi-algebraic assumptions. Our approach is mainly based on  asymptotic analysis which has widely been used in optimization problems, variational inequalities, complementarity problems, and equilibrium problems. See e.g. \cite{HV3,Lop,Gowda,Huangzh,FlorJOGO,FlorESAIM, FlorJCA}.

\par In this paper, the existence theorems of Pareto efficient solutions for $\text{PVOP}(K, f)$ are obtained under the relative (regularity / non-regularity) conditions. 	
Our main contributions are the following:
 {\begin{itemize}
  	\item  In \cite{Flores24,DTN,LGJ}, at least one of the convexity and coercivity conditions is supposed to obtain existence results for (Pareto / weak Pareto) efficient solutions of vector optimization problems. However, in this paper, we study polynomial vector optimization problems  with an arbitrary nonempty closed constraint set without any convexity, and coercivity assumptions.
 	
 \item  Existence results of weak Pareto efficient solutions obtained in \cite{Flores24} are obtained without any convexity and coercivity assumptions. In \cite{LDY1}, they obtained existence of Pareto efficient solutions for $\text{PVOP}(K, f)$ under the regularity conditions and boundedness from below condition. However, in this paper, we obtain nonemptiness of Pareto efficient solution set of the polynomial vector optimization problems under weaker regularity conditions and section-boundedness from below condition.
	
\item In this paper, we study the local properties of the relative regularity conditions and obtain genercity principle of the some relative regularity conditions. We extend and improve the corresponding results of \cite{HV1}.  Compared with \cite{DTN,LGJ1,LDY2}, our approach is mainly based on tools of asymptotic analysis, but not semi-algebraic theorem.
 \end{itemize}}

\par The rest of this paper is structured as follows: In Section 2, we present some fundamental notations and preliminary results essential for subsequent analysis.  In Section 3, we  systematically investigates key properties and characterizations of the relative regularity conditions. In Section 4, we establish and analyze the interconnections between the relative regularity conditions, Palais-Smale condition, weak Palais-Smale condition, M-tameness and properness with respect to some index set. Section 5 is devoted to study the existence of Pareto efficient solutions of ${\rm{PVOP}}(K, f)$ respectively under the relative regularity and non-regularity conditions. In Section 6, we discuss the local properties of relative regularity conditions and establish the genericity of relative regularity conditions within appropriate function spaces. Finally, we make a conclusion in Section 7.

\section{Preliminaries}

In this section, we recall some concepts and results. A nonempty subset $D\subseteq\mathbf{R}^{n}$ is called a cone, if $tx\in D$ for any $x\in D$ and $t>0$. Given a nonempty closed set $K\subset  \mathbf{R}^{n}$, the \emph{asymptotic cone} $K_{\infty}$  of $K$ is defined by
$$K_{\infty}=\{v\in \mathbf{R}^{n}: \mbox{there}\ \mbox{exist}\ t_{k}\rightarrow +\infty\ \mbox{and}\ x_{k}\in K\ \mbox{such}\ \mbox{that}\ \lim_{k\rightarrow +\infty}\frac{x_{k}}{t_{k}}=v\}.$$
As known, $K_{\infty}$ is a closed cone and $(K_{\infty})_{\infty}=K_{\infty}$, and $K$ is bounded if and only if $K_{\infty}=\{0\}$. These results can be found in  \cite{RC,AA}. Let $\bar x\in \mathbf{R}^{n}$. Then the sublevel set $K_{\bar x}$ is defined by $K_{\bar x}=\{x\in K\vert f(x)\leq f(\bar x)\}$.

\begin{definition}(\cite[Definition 2.1]{LDY1})
Let $p=(p_{1}, \dots, p_{q}): \mathbf{R}^{n} \mapsto \mathbf{R}^{q}$ be a vector polynomial  with $  \mbox{ deg }p_{i}=d_{i}$, $i=1,  \dots, q$. We say that  $p^{\infty}_{\mathbf{d}}$ is the \emph{vector recession polynomial} (or the \emph{vector leading term}) of $p$, where $\mathbf{d}=(d_1, d_2,\dots,d_q)$,
$$p_\mathbf{d}^{\infty}(x)=((p_{1})^{\infty}_{d_{1}}(x), (p_{1})^{\infty}_{d_{2}}(x), \dots, (p_{1})^{\infty}_{d_{q}}(x)) \emph{ and } (p_{1})^{\infty}_{d_{i}}(x)=\lim_{\lambda\to +\infty}\frac{p_{i}(\lambda x)}{\lambda^{d_{i}}},\quad\forall x\in \mathbf{R}^{n}.$$

\end{definition}

\begin{remark}\label{recpoly}
	When $q=1$, $p^{\infty}$ is just a recession polynomial of $p$ (see \cite{HV1}).
\end{remark}

\begin{definition}(\cite[Definition 3.3]{DTN})\label{strong section}
	Let $C\subseteq\mathbf{R}^q$ be a subset and $\bar t\in \mathbf{R}^q$. The set $C\bigcap (\bar t -\mathbf{R}^q_+)$ is called a section of $C$ at $\bar t$ and denoted by $\lbrack C \rbrack_{\bar t}$. The section $\lbrack C \rbrack_{\bar t}$ is said to be bounded if there exists $r\in \mathbf{R}^q$ such that
	$$\lbrack C \rbrack_{\bar t}\subseteq r +\mathbf{R}^q_+.$$
\end{definition}

\begin{definition} (\cite[Definition 2.2]{LDY2})\label{strong section-boundedness}
	Let $x'\in C$. A vector-valued function $T: \mathbf{R}^{n}\rightarrow \mathbf{R}^{q}$ is said to be \emph{section-bounded from below} at $x'$, if the section $\lbrack T(C) \rbrack_{T(x')}$ is bounded.
\end{definition}

By Definition \ref{strong section-boundedness},  a vector-valued function $T$ is section-bounded from below at $x'\in C$ if and only if there exists $r=(r_1, r_2, \cdots, r_q)\in \mathbf{R}^q$ such that
$$T_{i}(x)\geq r_i$$
for any $x\in C$ satisfying with $T_i(x)\leq T_i(x'), i\in \{1, 2, \cdots, q\}$. In \cite{DTN,LGJ1}, the section-boundedness from below has been used to prove the existence of efficient solutions for polynomial vector optimization problems.

Next, we recall that  the definition of the weak section-boundedness from below.

\begin{definition}(\cite[Definition 2.3]{LDY2})\label{section-boundedness}	
A vector-valued function $T=(T_{1}, T_{2}, \cdots, T_{q}): \mathbf{R}^{n}\rightarrow \mathbf{R}^{q}$ is said to be \emph{weakly section-bounded from below} on $C$,  if there exist $\bar x\in C$ and $\bar a\in \mathbf{R}^{q}$ such that
$$T(x)-\bar a\notin -{\rm{ int }}\mathbf{R}^{q}_{+},\quad\forall x\in C_{\bar x},$$
where $C_{\bar x}=\{x\in C: T(x)\leq T(\bar x)\}$.
\end{definition}

\begin{remark}\label{strong section compare}
    By Definitions \ref{strong section-boundedness} and \ref{section-boundedness}, the section-boundedness from below implies the weak section-boundedness from below. The inverse is not true in general. By \cite[Proposition 3]{LDY2}, we know that a equivalent characterization of the weak section-boundedness from below on $C$ of $T$ has been given, i.e, $T$ is weakly section-bounded from below on $C$ if and only if there exist $\bar x\in C$ and $i_{0}\in \{1, 2, \cdots, q\}$ such that $T_{i_{0}}$ is bounded from below on $C_{\bar x}$. Motivated by the above discussions, we now propose the following definition.
\end{remark}

\begin{definition}\label{exchangesec}
	Let $\bar x\in C$. A vector-valued function $T=(T_{1}, T_{2}, \cdots, T_{q}): \mathbf{R}^{n}\rightarrow \mathbf{R}^{q}$ is said to be \emph{$I$-section-bounded from below} at $\bar x$, if there exists a nonempty index set $I\subseteq\{1,2,\dots,q\}$ such that for any $i\in I$, $T_{i}$ is bounded from below on $C_{\bar x}$.
\end{definition}

\begin{remark}\label{20250227a}
By Definition \ref{exchangesec}, if $I=\{1,2,\dots,q\}$, then the $I$-section-boundedness from below reduces to section-boundedness from below. If $I\subsetneqq\{1,2,\dots,q\}$, then the $I$-section-boundedness from below reduces to weak section-boundedness from below.
\end{remark}

Now, we recall that the vector polynomial $f$ is said to be \emph{the strongly regular} (resp. \emph{the weakly regular}) on $K$, if $SOL^{w}(K_\infty, f^\infty_\mathbf{d})$ (resp. $SOL^{s}(K_\infty, f^\infty_\mathbf{d})$) is bounded (see. \cite[Definition 2.3]{LDY1}). When $q=1$, the scalar polynomial $f$ is said to be \emph{the regular} on $K$, that is,  $SOL(K_\infty, f^\infty_d)$ is bounded (see, \cite[Definition 2.1]{HV1}).

To explore polynomial vector optimization problems under relaxed regularity assumptions, we develop  regularity criteria associated with the asymptotic cone $K_\infty$.  Let $\bar x\in K$ and $K_{\bar x}=\{x\in K: f_i(x)\leq f_i(\bar x), i=1,2,\dots, q\}$. Consider a nonempty closed set $S\subseteq\mathbf{R}^n$ such that $(K_{\bar x})_\infty\subseteq S_\infty \subseteq K_\infty$. In general, such a set $S$ can be found. For example, if $K_{\bar x}\subseteq S \subseteq K$, then $(K_{\bar x})_\infty\subseteq S_\infty \subseteq K_\infty$. And we can also easy to prove $(K_{\bar x})_\infty\subseteq K_\infty\bigcap\{x\in\mathbf{R}^n\mid f(x)\leq f(\bar x)\}_\infty\subseteq K_\infty\bigcap\{x\in \mathbf{R}^n\mid f^\infty_\mathbf{d}(x)\leq 0\}\subseteq K_\infty$. Indeed, by \cite[Proposition 3.9]{RC}, the first inclusion relation holds. In particular, when $f$ is a convex mapping and $K$ is a convex set, we know that the first inclusion relation as an equation. Next, we claim that the second inclusion relation holds. Indeed, let $v\in K_\infty\bigcap\{x\in\mathbf{R}^n\mid f(x)\leq f(\bar x)\}_\infty$. Then there exist $v_k\in \{x\in\mathbf{R}^n\mid f(x)\leq f(\bar x)\}$ and $\lambda_k>0$ with $\lambda_k\to +\infty$ as $k\to +\infty$ such that $\frac{v_k}{\lambda_k}\to v$ as $k\to +\infty$. Since $v_k\in \{x\in\mathbf{R}^n\mid f(x)\leq f(\bar x)\}$, we have $f_i(v_k)\leq f_i(\bar x)$ for each $i\in \{1,2,\dots,q\}$. Dividing the both sides of the this inequality by $\lambda_k^{d_i}$ and then letting $k\rightarrow +\infty$, we get $(f_i)^\infty_{d_i}(v)\leq 0$ for each $i\in \{1,2,\dots,q\}$. Thus, it follows from $v\in K_\infty$ that $v\in \{x\in K_\infty\mid f^\infty_\mathbf{d}(x)\leq 0\}$. So $K_\infty\bigcap\{x\in\mathbf{R}^n\mid f(x)\leq f(\bar x)\}_\infty\subseteq \{x\in K_\infty\mid f^\infty_\mathbf{d}(x)\leq 0\}$. In general, when $f$ is a convex mapping and $K$ is a convex set, the equation $K_\infty\bigcap\{x\in\mathbf{R}^n\mid f(x)\leq f(\bar x)\}_\infty=\{x\in K_\infty\mid f^\infty_\mathbf{d}(x)\leq 0\}$ may not hold. For example, let $f: \mathbf{R}^2\mapsto \mathbf{R}, f(x)=x^2_1+x_2$ and $K=\{(x_1,x_2)\in\mathbf{R}^2\mid 0\leq x_1, 0\leq x_2\}$. Clearly, $K_\infty\bigcap\{x\in\mathbf{R}^n\mid f(x)\leq f(\bar x)\}_\infty\subsetneqq\{x\in K_\infty\mid f^\infty_\mathbf{d}(x)\leq 0\}$. Thus, if $S_\infty=K_\infty\bigcap\{x\in\mathbf{R}^n\mid f(x)\leq f(\bar x)\}_\infty$ or $S_\infty=K_\infty\bigcap\{x\in \mathbf{R}^n\mid f^\infty_\mathbf{d}(x)\leq 0\}$, then we obtain  $(K_{\bar x})_\infty\subseteq S_\infty \subseteq K_\infty$. By the above discussions, we introduce the following definition.

\begin{definition}  We say that
\begin{itemize}
\item[(i)] the vector polynomial $f$ is \emph{relatively $I$-$\mathbf{R}^q_+$-zero-regular} with $S$ on $K$, if there exist $\bar x\in K$, a nonempty closed set $S\subseteq\mathbf{R}^n$ satisfying with $(K_{\bar x})_\infty\subseteq S_\infty\subseteq K_\infty$ and $\lambda=(\lambda_1,\lambda_2,\dots,\lambda_q)\in \mathbf{R}^q_+\backslash\{0\}$ with the index set $I=\{i\in \{1,2,\dots,q\}\vert\lambda_i\neq 0\}$ such that $f_\lambda=\sum_{i=1}^{q}\lambda_i f_i=\sum_{i\in I}\lambda_i f_i$ is \emph{regular} on $S$, that is, the solution set $SOL(S_\infty, \{\sum_{i=1}^{q}\lambda_i f_i\}^{\infty}_d)$ is bounded, where $d=\deg \sum_{i=1}^{q}\lambda_i f_i$. Otherwise, $f$ is said to be \emph{relatively $\mathbf{R}^q_+$-zero-non-regular} on $K$. In particular, if $I=\{1,2,\dots,q\}$, then we say that  the vector polynomial $f$ is \emph{relatively $\mathbf{R}^q_+$-zero-regular} with $S$ on $K$.
\item[(ii)] the vector polynomial $f$ is \emph{relatively weakly regular} with $S$ on $K$, if there exist $\bar x\in K$ and a nonempty closed set $S\subseteq\mathbf{R}^n$ satisfying with $(K_{\bar x})_\infty\subseteq S_\infty\subseteq K_\infty$ such that $f$ is \emph{weakly  regular} on $S$, that is, the solution set $SOL^{s}(S_{\infty}, f^{\infty}_{\mathbf{d}})$  is bounded. Otherwise, $f$ is said to be \emph{relatively weakly non-regular} on $K$.
\item[(iii)] the vector polynomial $f$ is \emph{ relatively strongly regular} with $S$ on $K$, if there exist $\bar x\in K$ and a nonempty closed set $S\subseteq\mathbf{R}^n$ satisfying with $(K_{\bar x})_\infty\subseteq S_\infty\subseteq K_\infty$  such that $f$ is \emph{strongly regular} on $S$, that is, the solution set $SOL^{w}(S_{\infty}, f^{\infty}_{\mathbf{d}})$ is bounded. Otherwise, $f$ is said to be \emph{relatively strongly non-regular} on $K$.
\end{itemize}
\end{definition}

\begin{remark}\label{compa}
 Clearly, $SOL(S_\infty, \{\sum_{i=1}^{q}\lambda_i f_i\}^{\infty}_d)\subseteq SOL^{w}(S_{\infty}, f^{\infty}_{\mathbf{d}})$ and $SOL^{s}(S_{\infty}, f^{\infty}_{\mathbf{d}})\subseteq SOL^{w}(S_{\infty}, f^{\infty}_{\mathbf{d}})$. So the relatively strong regularity implies the relatively weak regularity, and the relatively strong regularity implies relative $I$-$\mathbf{R}^q_+$-zero-regularity. By \cite{HV1}, we know that relative $I$-$\mathbf{R}^q_+$-zero-regularity is equivalent to $SOL(S_\infty, \{\sum_{i=1}^{q}\lambda_i f_i\}^{\infty}_d)=\{0\}$ or $SOL(S_\infty, \{\sum_{i=1}^{q}\lambda_i f_i\}^{\infty}_d)=\emptyset$  for some $\lambda=(\lambda_1,\lambda_2,\dots,\lambda_q)\in \mathbf{R}^q_+\backslash\{0\}$.  By Proposition 3.3 and Remark 3.1 in \cite{LDY1}, we know that \emph{ relatively weak regularity} (resp.  \emph{ relatively strong regularity}) with $S$ on $K$ of $f$ is equivalent to $SOL^{s}(S_{\infty}, f^{\infty}_{\mathbf{d}})=\{0\}$ or $SOL^{s}(S_{\infty}, f^{\infty}_{\mathbf{d}})=\emptyset$ (resp. $SOL^{w}(S_{\infty}, f^{\infty}_{\mathbf{d}})=\{0\}$ or $SOL^{w}(S_{\infty}, f^{\infty}_{\mathbf{d}})=\emptyset$).
\end{remark}

\begin{remark}
 When $q=1$, we say that the relative $I$-$\mathbf{R}^q_+$-zero-regularity, relatively weak regularity and relatively strong regularity are \emph{relative regularity}. In particular, if let $S=K_\infty$, then all the relative regularity conditions coincide with the regularity condition. If $f$ is bounded from below on $K$, then the relative regularity condition is weaker than the regularity condition. Indeed, when $f$ is bounded from below on $K$, if $f$ is regular on $K$, then we know $SOL(K_\infty, f^\infty_d)=\{0\}$ by \cite{HV1}. Let $S=K_\infty$. Then $S_\infty=S$. Thus, $f$ is relatively regular with $S$ on $K$. However, the following example shows that its inverse may not true.
\begin{example}
 	 Consider the polynomial $f: \mathbf{R}^2\mapsto \mathbf{R}, f(x_1, x_2)=x^4_1+x^2_2$ and $K=\mathbf{R}^2$. Clearly, $f$ is bounded from below on $K$, $K_{\infty}=K$, and $f^\infty_d(x_1, x_2)=x^4_1$. On the one hand, we know that $SOL(K_{\infty}, f^{\infty}_d)=\{(x_1,x_2)\in \mathbf{R}^2: x_1=0\}$, which is an unbounded set. Thus, $f$ is non-regular on $K$. On the other hand, let $S=K_{\infty}\bigcap\{x=(x_1,x_2)\in \mathbf{R}^2\mid f(x_1, x_2)\leq f(0,0)\}_\infty$. Clearly, $S$ is a nonempty closed cone and $(K_{\bar x})_\infty\subseteq S_\infty\subseteq K_\infty$ for any $\bar x\in K$. It is easy to calculate $S_\infty=\{(0, 0)\}$. And so, $SOL(S_{\infty}, f^{\infty}_{\mathbf{d}})=\{(0,0)\}$. Thus, $f$ is relatively regular with $S$ on $K$.
\end{example}
 Next, we recall that the scalar mapping $f$ is said to be coercive on $K$, if $\lim_{x\in K, \|x\|\to +\infty} f(x)=+\infty$. Let $\bar x\in K$ and $S=K_{\bar x}$. So $(K_{\bar x})_\infty\subseteq S_\infty\subseteq K_\infty$. It is easy to prove that if the scalar function $f$ is bounded from below on $K$, then the coercivity on $K$ of $f$ is equivalent to the relative regularity with $S$ on $K$ of $f$. Indeed, if $f$ is coercive on $K$, then $S$ is bounded. So $S_\infty=(K_{\bar x})_\infty=\{0\}$. Thus, we have $SOL(S_{\infty}, f^{\infty}_d)=\{0\}$. So $f$ is relatively regular with $S$ on $K$. Conversely, suppose on the contrary that $f$ is not coercive on $K$. Then there exists a sequence $\{x_k\}\subseteq \{x\in K\mid f(x)\leq f(\bar x)\}$ such that $\|x_k\|\to +\infty$ as $k\to +\infty$. Assume that $\frac{x_k}{\|x_k\|}\to v_0$ as $k\to +\infty$. Then $v_0\in S_\infty=(K_{\bar x})_\infty\setminus\{0\}$ and dividing the both sides of the inequality $f(x_k)\leq f(\bar x)$ by ${x_k}^{d}$ with $d=\deg f$ and then letting $k\rightarrow +\infty$, we get $f^\infty_d(v_0)\leq 0$. By \cite{HV1} and boundedness from below on $K$ of $f$, we have $f^\infty_d\geq 0$ on $K_\infty$. It follows that $v_0\in SOL(S_{\infty}, f^{\infty}_d)\setminus\{0\}$, which is a contradiction with the relative regularity with $S$ on $K$. The following example shows that the condition of boundedness from below on $K$ of $f$ can not drop.

 \begin{example}
    Consider the polynomial $f: \mathbf{R}\mapsto \mathbf{R}, f(x)=x$ and $K=\mathbf{R}$. Clearly, $f$ is not bounded from below on $K$ and $f^{\infty}_d=f$. Let $\bar x\in K$ and $S={K_{\bar x}}$. it is easy to prove $SOL(S_{\infty}, f^{\infty}_d)=\emptyset$. Thus, $f$ is relatively  regular with $S$ on $K$. However, it is clear that $f$ is not coercive on $K$.
 \end{example}

It is notice that when $K$ is a convex set and $f$ is a convex mapping, we let $\bar x\in K$ and $S=K_\infty\bigcap\{x\in\mathbf{R}^n\mid f(x)\leq f(\bar x)\}_\infty$. Then $(K_{\bar x})_\infty=K_\infty\bigcap\{x\in\mathbf{R}^n\mid f(x)\leq f(\bar x)\}_\infty=S_\infty$. Thus, if $f$ is bounded from below on $K$, then the coercivity on $K$ of $f$ is also equivalent to the relative regularity with $S$ on $K$ of $f$.
\end{remark}

\begin{remark}
 When $q\geq 2$, in \cite[Definition 3.1]{GGGG}, $f$ is said to be $\mathbf{R}^q_+$-zero-coercive on $K$ with respect to $\alpha\in \mathbf{R}^q_+\backslash\{0\}$, if $\lim_{x\in K, \|x\|\to +\infty} \langle \alpha, f(x)\rangle=+\infty$. We know that the relatively (weak / strong) regularity with $S$ on $K$ of $f$ is weaker than the $\mathbf{R}^q_+$-zero-coercivity on $K$ of $f$. Indeed, let $\bar x$ and $S=K_{\bar x}$. If $f$ is $\mathbf{R}^q_+$-zero-coercive on $K$, then for any the sequence $\{x_k\}\subseteq K$ with $\|x_k\|\to +\infty$, there exists $i_0\in\{1,2, \dots,q\}$ such that $f_{i_0}(x_k)\to +\infty$. So $K_{\bar x}$ is a bounded set. Thus, $S_{\infty}=\{0\}$. It follows that $SOL^{s}(S_{\infty}, f^{\infty}_{\mathbf{d}})=\{0\}$ and $SOL^{w}(S_{\infty}, f^{\infty}_{\mathbf{d}})=\{0\}$. Thus, $f$ is relatively strong regular with $S$ on $K$, and so $f$ is relatively weak regular. However, the following example shows that its inverse may not hold in general.
 \begin{example}\label{250709}
    Consider the vector polynomial $f: \mathbf{R}\mapsto \mathbf{R}^2, f(x)=(-x^3, x^3)$ and $K=\mathbf{R}$. Then, let $\bar x=0\in K$ and $S=K_{\infty}\bigcap\{x\in \mathbf{R}\mid f^\infty_\mathbf{d}(x)\leq 0\}$. Then $S=S_\infty$ and $(K_{\bar x})_\infty\subseteq S_\infty\subseteq K_\infty$. It is easy to calculate that $S_\infty=\{0\}$. So $SOL^s(S_{\infty}, f^{\infty}_{\mathbf{d}})=\{0\}$ and $SOL^w(S_{\infty}, f^{\infty}_{\mathbf{d}})=\{0\}$. Thus, $f$ is both relatively weakly regular and relatively strongly regular with $S$ on $K$. However, it is clear that $f$ is not $\mathbf{R}^2_+$-zero-coercive on $K$.
 \end{example}
 \end{remark}

By the above Example \ref{250709}, it is notice that the relative $I$-$\mathbf{R}^q_+$-zero-regularity is also weaker than the $\mathbf{R}^q_+$-zero-coercivity.

\section{Characteristics and properties of the relative regularity conditions}

In this section, we shall discuss the properties and characterizations of the relative regularity conditions. We  obtain some necessary conditions of existence of the Pareto efficient solutions of ${\rm{PVOP}}(K,f)$. We first give characterizations of $SOL^{s}(S_{\infty}, f^{\infty}_{\mathbf{d}})=\emptyset$ and $SOL^{w}(S_{\infty}, f^{\infty}_{\mathbf{d}})=\emptyset$.

\begin{proposition}\label{necessary}
	Let the nonempty closed set $S\subseteq \mathbf{R^n}$ satisfying with $S_\infty\subseteq \{x\in \mathbf{R}^n\mid f^\infty_{\mathbf{d}}(x)\leq 0\}$. Then the following conclusions hold:
	\begin{itemize}
	\item[(i)] $SOL^{s}(S_{\infty}, f^{\infty}_{\mathbf{d}})=\emptyset$  if and only if $0\notin SOL^{s}(S_{\infty}, f^{\infty}_{\mathbf{d}})$;
	\item[(ii)] $SOL^{w}(S_{\infty}, f^{\infty}_{\mathbf{d}})=\emptyset$  if and only if $0\notin SOL^{w}(S_{\infty}, f^{\infty}_{\mathbf{d}})$.
	\end{itemize}
\end{proposition}
\begin{proof}
\emph{(i)}: We only need to prove the sufficiency. Since $S_\infty\subseteq \{x\in \mathbf{R}^n\mid f^\infty_{\mathbf{d}}(x)\leq 0\}$, we can easy to prove $S_{\infty}=\{x\in S_\infty\mid f^\infty_{\mathbf{d}}(x)\leq 0\}$. Suppose on the contrary that $SOL^{s}(S_{\infty}, f^{\infty}_{\mathbf{d}})\neq\emptyset$. Since $0\notin SOL^{s}(S_{\infty}, f^{\infty}_{\mathbf{d}})$ and $f^{\infty}(0)=0$, there exists $v_{1}\in SOL^{s}(S_{\infty}, f^{\infty}_{\mathbf{d}})$ such that $f^{\infty}(v_1)\neq 0$. Because $S_{\infty}=\{x\in S_\infty\mid f^\infty_{\mathbf{d}}(x)\leq 0\}$, we have $v_{1}\in SOL^{s}(\{x\in S_\infty\mid f^\infty_{\mathbf{d}}(x)\leq 0\}, f^{\infty}_{\mathbf{d}})$. By $f^{\infty}(v_1)\neq 0$ and  $v_{1}\in \{x\in S_\infty\mid f^\infty_{\mathbf{d}}(x)\leq 0\}$, we have $f^\infty_i(v_1)\leq 0$ for all $i=1,2,\dots,i_0-1,i_0+1,\dots,q$ and $f^{\infty}_{i_{0}}(v_1)< 0$ for some $i_0$. It follows that for all $t>1$,
$$f^\infty_i(tv_1)-f^\infty_i(v_1)\leq 0 \ and \ f^{\infty}_{i_{0}}(tv_1)-f^{\infty}_{i_{0}}(v_1)< 0$$
for all $i=1,2,\dots,i_0-1,i_0+1,\dots,q$. Since $tv_1\in S_{\infty}$, we have  $v_{1}\notin SOL^{s}(S_{\infty}, f^{\infty}_{\mathbf{d}})$, which is a contradiction with $v_{1}\in SOL^{s}(S_{\infty}, f^{\infty}_{\mathbf{d}})$.


\emph{(ii)}: By \cite[Proposition 3.1]{LDY1}, this result can be obtained, directly.
\end{proof}

 Now, we give an example to illustrate Proposition \ref{necessary}.
\begin{example}\label{emptyex}
	Consider the vector polynomial $f=(f_{1}, f_{2})$  with
$$f_{1}(x_{1}, x_{2})=x_{1}, f_{2}(x_{1}, x_{2})=x_{2}$$
and
$$K=\mathbf{R}^2.$$
Let $\bar x=(0, 0)$. It is easy to verify that $(K_{\bar x})_{\infty}=\{(x_{1}, x_{2})\in \mathbf{R}^{2}: x_{1}\leq 0, x_{2}\leq 0\}$,  $(f_{1})^{\infty}_{d_1}(x_{1}, x_{2})=x_{1}$, and $(f_{2})^{\infty}_{d_2}(x_{1}, x_{2})=x_{2}$. Let $S=(K_{\bar x})_{\infty}$. Then $S_\infty\subseteq\{x\in \mathbf{R}^2\mid f^\infty(x)\leq 0\}$. So $0\notin SOL^{s}(S_{\infty}, f^{\infty}_{\mathbf{d}})$ and $0\notin SOL^{w}(S_{\infty}, f^{\infty}_{\mathbf{d}})$ since $(f_{1})^{\infty}_{d_1}(-1, -1)=(f_{2})^{\infty}_{d_2}(-1, -1)=-1<0=(f_{1})^{\infty}_{d_1}(0, 0)=(f_{2})^{\infty}_{d_2}(0, 0)$. By Proposition \ref{necessary}, $SOL^{s}(S_{\infty}, f^{\infty}_{\mathbf{d}})=SOL^{w}(S_{\infty}, f^{\infty}_{\mathbf{d}})=\emptyset$.
\end{example}

\begin{proposition}\label{notbounded}
	Let the nonempty set $S\subseteq \mathbf{R^n}$. Then the following results hold:
	\begin{itemize}
		\item[(i)] If $SOL^{w}(S_{\infty}, f^{\infty}_{\mathbf{d}})=\emptyset$, then $f_{i}$ is unbounded from below on $S$ for all $i\in \{1, \dots, q\}$.
		\item[(ii)] If $SOL^{s}(S_{\infty}, f^{\infty}_{\mathbf{d}})=\emptyset$, then there exists $i_0\in\{1,2, \dots, q\}$ such that $f_{i_0}$ is unbounded from below on $S$.
	\end{itemize}
\end{proposition}
\begin{proof}
(\emph{i}): The first result  follows from \cite[Proposition 3.4]{LDY1}.

(\emph{ii}): Suppose on the contrary that $f_{i}$ is bounded from below on $S$ for all $i\in \{1, \dots, q\}$. Then, there exist $r_i\in \mathbf{R}, i=1,2, \dots,q$ such that $$f_i(x)\geq r_i,$$
for any $x\in S$. Let $v\in S_\infty$ be arbitrary. Then there exist $t_k >0$ with $t_{k}\rightarrow +\infty$ and $x_{k}\in S$ such that $t^{-1}_{k}x_{k}\rightarrow v_{0}$ as $k\rightarrow +\infty$. Since $$f_i(x_k)\geq r_i,$$
for each $i=1,2,\dots,q$ and all $k$. Dividing the both sides of the above inequality by $t_k$ and then letting $k\rightarrow +\infty$, we get
$$(f_i)^\infty_{d_i}(v)\geq 0,$$
for each $i=1,2,\dots,q$. Thus, by the arbitration of $v$, we have $0\in SOL^{s}(S_{\infty}, f^{\infty}_{\mathbf{d}})$, which is a contradiction.
\end{proof}

In particular, let $\bar x\in K$ and $S=(K_{\bar x})_\infty$. Then we have the following results.

\begin{corollary}\label{weakpro}
Let $\bar x\in K$ and the nonempty index set $I\subseteq\{1,2,\dots,q\}$. If $f$ is $I$-section-bounded from below at $\bar x$, then $0\in SOL^{w}((K_{\bar x})_{\infty}, f^{\infty}_{\mathbf{d}})$. In particular, if $f$ is section-bounded from below at $\bar x$, then $0\in SOL^{s}((K_{\bar x})_{\infty}, f^{\infty}_{\mathbf{d}})$.
\end{corollary}

\begin{proof}
	Since $f$ is $I$-section-bounded from below at $\bar x$, $f_{i_0}$ is bounded from below on $K_{\bar x}$ for any $i_0\in I$. By  Proposition \ref{notbounded} \emph{(i)}, we have $SOL^{w}((K_{\bar x})_{\infty}, f^{\infty}_{\mathbf{d}})\neq\emptyset$. Thus, by Proposition \ref{necessary} \emph{(ii)}, we know $0\in SOL^{w}((K_{\bar x})_{\infty}, f^{\infty}_{\mathbf{d}})$. In particular, if $f$ is section-bounded from below at $\bar x$, then $f_{i}$ is bounded from below on $K_{\bar x}$ for all $i\in \{1, \dots, q\}$. By Proposition \ref{notbounded} \emph{(ii)}, we have $SOL^{s}((K_{\bar x})_{\infty}, f^{\infty}_\mathbf{d})\neq\emptyset$. Thus, by Proposition \ref{necessary} \emph{(i)}, we know $0\in SOL^{s}((K_{\bar x})_{\infty}, f^{\infty}_\mathbf{d})$.
\end{proof}

\begin{remark}
Let $\bar x\in K$. When $f$ is a convex mapping on $K$ and $K$ is a convex set, we know $(K_{\bar x})_{\infty}=K_\infty\bigcap\{x\in \mathbf{R}^n\mid f(x)\leq f(\bar x)\}_\infty$. So, by Corollary \ref{weakpro}, we have the following results.
\end{remark}

\begin{corollary}\label{2410121}
Assume that $K$ is a convex set and $f$ is a convex polynomial mapping on $K$. Let $S=K_\infty\bigcap\{x\in \mathbf{R}^n\mid f(x)\leq f(\bar x)\}_\infty$ with $\bar x\in K$. If $f$ is $I$-section-bounded from below at $\bar x$, then $0\in SOL^{w}(S_{\infty}, f^{\infty}_{\mathbf{d}})$. In particularly, if $f$ is section-bounded from below at $\bar x$, then $0\in SOL^{s}(S_{\infty}, f^{\infty}_{\mathbf{d}})$.
\end{corollary}

\begin{proposition}\label{2410234}
	Let the nonempty set $S\subseteq \mathbf{R^n}$ and the nonempty index set $I\subseteq\{1,2,\dots,q\}$ denoted by $I=\{s_1, s_2,\dots, s_p\}$. Assumed that $f_I=(f_{s_1},f_{s_2},\dots, f_{s_p})$ is bounded from below on $S$. Then there exists  $\lambda=(\lambda_1,\lambda_2,\dots,\lambda_q)\in \mathbf{R}^q_+\setminus\{0\}$ such that $0\in SOL(S_{\infty}, \{\sum_{i=1}^{q}\lambda_i f_i\}^{\infty}_d)$.
\end{proposition}
\begin{proof}
Since $f_I$ is bounded from below on $S$, there exists $r_i\in \mathbf{R}$ such that $r_{i}\leq f_{q_i}(x)$ for any $i\in I$ and $x\in S$. Let $v\in S_{\infty}$ be arbitrary. Then there exist $t_k>0$ with $t_{k}\rightarrow +\infty$ as $k\to +\infty$ and $x_{k}\in S$ such that $t^{-1}_{k}x_{k}\rightarrow v_{0}$ as $k\rightarrow +\infty$. Let $\lambda=(\lambda_1,\lambda_2,\dots,\lambda_q)$ with $\lambda_i=0, i\in \{1,2,\dots,q\}\setminus I$ and $\lambda_j=1, j\in I$. Since $f_{i}(x_k)\geq r_{i}$ for each $i\in I$ and all $k$, we have $\sum_{i=1}^{q}\lambda_i r_{i}\leq \sum_{i=1}^{q}\lambda_i f_{i}(x_k)$. Dividing the both sides of the previous inequality by $(t_k)^d$, where $d=\max_{j\in I}\{deg f_j\}$ and letting $k\rightarrow +\infty$, we get $0\leq\{\sum_{i=1}^{q}\lambda_i f_i\}^{\infty}_d(v)$. It follows from $\{\sum_{i=1}^{q}\lambda_i f_{i}\}^{\infty}_d(0)=0$ and the arbitrariness of $v\in S_{\infty}$ that $0\in SOL(S_{\infty}, \{\sum_{i=1}^{q}\lambda_i f_i\}^{\infty}_d)$.
\end{proof}

When $I=\{1,2,\dots,q\}$, by the same with the proof of Proposition \ref{2410234}, we can easy to get the following result.

\begin{proposition}\label{2410141}
	Let the nonempty set $S\subseteq \mathbf{R^n}$. Assumed that $f$ is bounded from below on $S$. Then $0\in SOL(S_{\infty}, \{\sum_{i=1}^{q}\lambda_i f_i\}^{\infty}_d)$ for any  $\lambda=(\lambda_1,\lambda_2,\dots,\lambda_q)\in \mathbf{R}^q_+\setminus\{0\}$.
\end{proposition}

In particular, let $\bar x\in K$ and $S=K_{\bar x}$. Then, by Propositions \ref{2410234} and  \ref{2410141}, we can obtain the following result.

\begin{corollary}\label{2410235}
	Let $\bar x\in K$ and the nonempty index set $I\subseteq\{1,2,\dots,q\}$. Assumed that $f$ is $I$-section-bounded from below at $\bar x$. Then there exists  $\lambda=(\lambda_1,\lambda_2,\dots,\lambda_q)\in \mathbf{R}^q_+\setminus\{0\}$ such that $0\in SOL((K_{\bar x})_{\infty}, \{\sum_{i=1}^{q}\lambda_i f_i\}^{\infty}_d)$. In particular, if $f$ is section-bounded from below at $\bar x$, then $0\in SOL((K_{\bar x})_{\infty}, \{\sum_{i=1}^{q}\lambda_i f_i\}^{\infty}_d)$ for any  $\lambda=(\lambda_1,\lambda_2,\dots,\lambda_q)\in \mathbf{R}^q_+\setminus\{0\}$.
\end{corollary}

\begin{remark}
Let $\bar x\in K$. When $f$ is a convex mapping on $K$ and $K$ is a convex set, by Corollary \ref{2410235}, we have the following result.
\end{remark}

\begin{corollary}\label{2410236}
Assume that $K$ is a nonempty closed convex set and $f$ is a convex polynomial mapping on $K$. Let the nonempty index set $I\subseteq\{1,2,\dots,q\}$, $\bar x\in K$ and $S=K_\infty\bigcap\{x\in \mathbf{R}^n\mid f(x)\leq f(\bar x)\}_\infty$. If $f$ is $I$-section-bounded from below at $\bar x$, then there exists  $\lambda=(\lambda_1,\lambda_2,\dots,\lambda_q)\in \mathbf{R}^q_+\setminus\{0\}$ such that $0\in SOL(S_{\infty}, \{\sum_{i=1}^{q}\lambda_i f_i\}^{\infty}_d)$. In particular, if $f$ is section-bounded from below at $\bar x$, then $0\in SOL(S_{\infty}, \{\sum_{i=1}^{q}\lambda_i f_i\}^{\infty}_d)$ for any  $\lambda=(\lambda_1,\lambda_2,\dots,\lambda_q)\in \mathbf{R}^q_+\setminus\{0\}$.
\end{corollary}

Now, the following results show that the relative regularity conditions of $f$ is closely related to the relative regularity of $f_i, i\in \{1,2,\dots,q\}$. It plays an important role in investigating the existence of efficient solutions for polynomial vector optimization problems.

\begin{theorem}\label{2410122}Let the nonempty closed set $S\subseteq \mathbf{R^n}$. The following results are equivalent:
\begin{itemize}
	\item[(i)] $SOL^{w}(S_{\infty}, f^{\infty}_\mathbf{d})=\{0\}$;
    \item[(ii)] $SOL(S_{\infty}, (f_{i})^{\infty}_{d_i})=\{0\}$ for all $i\in \{1, \dots, q\}$;
    \item[(iii)] $SOL(S_{\infty}, \{\sum_{i=1}^{q}\lambda_i f_i\}^{\infty}_d)=\{0\}$ for any $\lambda=(\lambda_1,\lambda_2,\dots,\lambda_q)\in \mathbf{R}^q_+\setminus\{0\}$.
\end{itemize}
   Moreover, if $S\subseteq \mathbf{R^n}$ satisfies with $S_\infty\subseteq \{x\in \mathbf{R}^n\mid f^\infty(x)\leq 0\}$, then the conclusions (i)-(iii) are equivalent with the following result:
	\begin{itemize}
	\item[(iv)] $SOL^{s}(S_{\infty}, f^{\infty}_\mathbf{d})=\{0\}$.
    \end{itemize}
 In addition, if the one of the conditions (i)-(iv) holds, then there exists $\lambda=(\lambda_1,\lambda_2,\dots,\lambda_q)\in \mathbf{R}^q_+\setminus\{0\}$ such that $SOL(S_{\infty}, \{\sum_{i=1}^{q}\lambda_i f_i\}^{\infty}_d)=\{0\}$.
\end{theorem}
\begin{proof}
"\emph{(i)}$\Leftrightarrow$ \emph{(ii)}": By \emph{(i)} of \cite[Theorem 3.6]{LDY1}, this result is directly.

"\emph{(ii)}$\Leftrightarrow$ \emph{(iii)}": Assume that the conclusion \emph{(iii)} holds. Let $\lambda^{i_0}=(0,0,\dots,1,0\dots,0)$ with $\lambda_{i_0}=1$ and $\lambda_j=0$, $j\in\{1,2,\dots,i_0-1,i_0+1,\dots,q\}$, we have $SOL(S_{\infty}, (f_{i})^{\infty}_{d_i})=\{0\}$. Thus, the conclusion \emph{(ii)} holds. Conversely, if $SOL(S_{\infty}, (f_{i})^{\infty}_{d_i})=\{0\}$ for all $i\in \{1, \dots, q\}$, then the conclusion \emph{(iii)} holds, directly.

Moreover, if $S\subseteq \mathbf{R^n}$ satisfies with $S_\infty\subseteq \{x\in \mathbf{R}^n\mid f^\infty(x)\leq 0\}$, then we shall prove "\emph{(iv)}$\Leftrightarrow$ \emph{(ii)}". Assume that the conclusion \emph{(iv)} holds. By \emph{(ii)} of \cite[Theorem 3.6]{LDY1}, we have that for any $i\in\{1,2,\dots,q\}$, $SOL(S_{\infty}, (f_{i})^{\infty}_{d_i})\neq\emptyset$. Thus, $(f_{i})^{\infty}_{d_i}(v)\geq 0$ for any $v\in S_{\infty}$ and $i\in\{1,2,\dots,q\}$.  It follows from $S_\infty\subseteq \{x\in \mathbf{R}^n\mid f^\infty(x)\leq 0\}$  that $(f_{i})^{\infty}_{d_i}(v)=0$ for any $v\in S_{\infty}$ and $i\in\{1,2,\dots,q\}$. So $S_\infty=\{0\}$, since if there exists $v_0\in S_\infty\setminus\{0\}$, then $v_0\in SOL^{s}(S_{\infty}, f^{\infty}_\mathbf{d})\setminus\{0\}$ by $0\in SOL^{s}(S_{\infty}, f^{\infty}_\mathbf{d})$ and $(f_{i})^{\infty}_{d_i}(v_0)=0$ for all $i\in\{1,2,\dots,q\}$. Thus, the result \emph{(ii)} holds. Conversely, if the conclusion \emph{(ii)} holds, then we obtain $SOL^{s}(S_{\infty}, f^{\infty}_\mathbf{d})\neq\emptyset$. Thus, by the conclusion \emph{(i)} and $SOL^{s}(S_{\infty}, f^{\infty}_\mathbf{d})\subseteq SOL^{w}(S_{\infty}, f^{\infty}_\mathbf{d})$, we know the conclusion \emph{(iv)} holds.

In addition, If the one of the conclusions \emph{(i)}-\emph{(iv)} holds, then there exists $\lambda=(\lambda_1,\lambda_2,\dots,\lambda_q)\in \mathbf{R}^q_+\setminus\{0\}$ such that $SOL(S_{\infty}, \{\sum_{i=1}^{q}\lambda_i f_i\}^{\infty}_d)=\{0\}$.
\end{proof}

\begin{remark}
Assumed that the condition $S_\infty\subseteq \{x\in \mathbf{R}^n\mid f^\infty(x)\leq 0\}$ is removed. By the above proof in Theorem \ref{2410122}, we know that if the one of the conclusions (i)-(iii) in Theorem \ref{2410122} holds, then the conclusion (iv) is also true. However, the following example shows that its inverse may not hold in general without $S_\infty\subseteq \{x\in \mathbf{R}^n\mid f^\infty(x)\leq 0\}$ assumption.
\end{remark}

\begin{example}
	Consider the vector polynomial $f=(f_{1}, f_{2})$  with
$$f_{1}(x_{1}, x_{2})=x_{1}, f_{2}(x_{1}, x_{2})=x_{2}$$
and
$$K=\{(x_1, x_2)\in\mathbf{R}^2\mid 0\leq x_{1}, 0\leq x_{2}\}.$$
It is easy to verify that $(f_{1})^{\infty}_{d_1}(x_{1}, x_{2})=x_{1}$, and $(f_{2})^{\infty}_{d_2}(x_{1}, x_{2})=x_{2}$. Let $S=K$. Then $S_\infty=K_\infty$. Then $S_\infty\nsubseteq \{x\in \mathbf{R}^2\mid f^\infty(x)\leq 0\}$. Clearly, we can calculate $SOL^{s}(S_{\infty}, f^{\infty}_\mathbf{d})=\{(0, 0)\}$. However, $SOL^{w}(S_{\infty}, f^{\infty}_\mathbf{d})=SOL(S_{\infty}, (f_{i})^{\infty}_{d_i})=SOL(S_{\infty}, \{\sum_{i=1}^{q}\lambda_i f_i\}^{\infty}_d)=\{(x_1, x_2)\in\mathbf{R}^2\mid x_{1}=0, x_{2}\in\mathbf{R}\}\cup \{(x_1, x_2)\in\mathbf{R}^2\mid x_{2}=0, x_{1}\in\mathbf{R}\}$ for all $i\in \{1, 2\}$ and $\lambda=(\lambda_1,\lambda_2)\in \mathbf{R}^2_+\setminus\{(0, 0)\}$. This means that the conclusions (i)-(iii) in Theorem \ref{2410122}  are not valid.
\end{example}

\begin{remark}
It is notice that the following example shows that if there exists $\lambda=(\lambda_1,\lambda_2,\dots,\lambda_q)\in \mathbf{R}^q_+\setminus\{0\}$ such that $SOL(S_{\infty}, \{\sum_{i=1}^{q}\lambda_i f_i\}^{\infty}_d)=\{0\}$, then conditions (i)-(iv) in Theorem \ref{2410122} may not hold.
\end{remark}

\begin{example}
Consider the vector polynomial $f=(f_{1}, f_{2})$  with
$$f_{1}(x_{1}, x_{2})=x^2_{1}x_{2}+x_{1}, f_{2}(x_{1}, x_{2})=-x^2_{1}x_{2}+x_{2}$$
and
$$K=\{(x_1, x_2)\in\mathbf{R}^2\mid 0\leq x_{1}, 0\leq x_{2}\}.$$
Clearly, $(f_{1})^{\infty}_{d_1}(x_{1}, x_{2})=x^2_{1}x_{2}$ and $(f_{2})^{\infty}_{d_2}(x_{1}, x_{2})=-x^2_{1}x_{2}$. Let $\bar x=(0, 0)\in K$ and $S=K_\infty\bigcap\{x\in \mathbf{R}^2\mid f^\infty(x)\leq 0\}$. Then $S_\infty=\{(x_1,x_2)\in\mathbf{R}^2\mid x_{1}=0, 0\leq x_{2}\}\bigcup \{(x_1,x_2)\in\mathbf{R}^2\mid x_{2}=0, 0\leq x_{1}\}$. Let $\lambda=(1,1)\in \mathbf{R}^2_+\setminus\{(0, 0)\}$. Then we have get $SOL(S_{\infty}, \{\sum_{i=1}^{2}\lambda_i f_i\}^{\infty}_d)=\{(0, 0)\}$. However, $SOL(S_{\infty}, f_2^{\infty})=S_\infty\neq \{(0, 0)\}$. Thus, the conditions (iii) in Theorem \ref{2410122} does not hold. And so,  the conclusions (i)-(ii) and (iv) in Theorem \ref{2410122} also do not hold.
\end{example}

In particular, let $\bar x\in K$ and $S=K_{\bar x}$. Since $(K_{\bar x})_\infty\subseteq \{x\in \mathbf{R}^n\mid f^\infty(x)\leq 0\}$, by Theorem \ref{2410122}, we have the following result.
\begin{corollary}\label{gpro}Let $\bar x\in K$. The following results are equivalent:
\begin{itemize}
	\item[(i)] $SOL^{s}((K_{\bar x})_{\infty}, f^{\infty}_\mathbf{d})=\{0\}$;
	\item[(ii)] $SOL^{w}((K_{\bar x})_{\infty}, f^{\infty}_\mathbf{d})=\{0\}$;
	\item[(iii)]$SOL((K_{\bar x})_{\infty}, (f_{i})^{\infty}_{d_i})=\{0\}$ for all $i\in \{1, \dots, q\}$.
    \item[(iv)] $SOL((K_{\bar x})_{\infty}, \{\sum_{i=1}^{q}\lambda_i f_i\}^{\infty}_d)=\{0\}$ for any $\lambda=(\lambda_1,\lambda_2,\dots,\lambda_q)\in \mathbf{R}^q_+\setminus\{0\}$.
\end{itemize}
    Moreover, If the one of the conditions (i)-(iv) holds, then there exists $\lambda=(\lambda_1,\lambda_2,\dots,\lambda_q)\in \mathbf{R}^q_+\setminus\{0\}$ such that $SOL((K_{\bar x})_{\infty}, \{\sum_{i=1}^{q}\lambda_i f_i\}^{\infty}_d)=\{0\}$.
\end{corollary}

\begin{remark}
\cite[Theorem 3.6]{LDY1} shows that $SOL^{s}(K_{\infty}, f^{\infty}_\mathbf{d})=\{0\}$ implies  $SOL(K_{\infty}, (f_{i})^{\infty}_{d_i})\neq\emptyset$ for each $i\in \{1,  \dots, q\}$. However, \cite[Example 3.6]{LDY1} shows that $SOL^{s}(K_{\infty}, f^{\infty}_\mathbf{d})=\{0\}$ does not imply  $SOL(K_{\infty}, (f_{i})^{\infty}_{d_i})= \{0\}$ for each $i\in \{1,  \dots, q\}$. As a comparison, Theorem \ref{2410122} and Corollary \ref{gpro} show that $SOL^{s}(S_{\infty}, f^{\infty}_\mathbf{d})=\{0\}$ is equivalent with $SOL(S_{\infty}, (f_{i})^{\infty}_{d_i})=\{0\}$ for all $i\in \{1, \dots, q\}$ with $S_\infty\subseteq \{x\in \mathbf{R}^n\mid f^\infty_\mathbf{d}(x)\leq 0\}$.
\end{remark}

When $f$ is a convex polynomial mapping on $K$ and $K$ is a convex set, by Corollary \ref{gpro}, we also have the following result.

\begin{corollary}
 Assume that $K$ is a convex set and $f$ is a convex polynomial mapping on $K$. Let $S=K_\infty\bigcap\{x\in \mathbf{R}^n\mid f(x)\leq f(\bar x)\}_\infty$ with $\bar x\in K$.
Then the following results are equivalent:
	\begin{itemize}
\item[(i)] $SOL^{s}(S_{\infty}, f^{\infty}_\mathbf{d})=\{0\}$;
	\item[(ii)] $SOL^{w}(S_{\infty}, f^{\infty}_\mathbf{d})=\{0\}$
    \item[(iii)] $SOL(S_{\infty}, (f_{i})^{\infty}_{d_i})=\{0\}$ for all $i\in \{1, \dots, q\}$.
    \item[(iv)] $SOL(S_{\infty}, \{\sum_{i=1}^{q}\lambda_i f_i\}^{\infty}_d)=\{0\}$ for any $\lambda=(\lambda_1,\lambda_2,\dots,\lambda_q)\in \mathbf{R}^q_+\setminus\{0\}$.
\end{itemize}
    Moreover, If the one of the conditions (i)-(iv) holds, then there exists $\lambda=(\lambda_1,\lambda_2,\dots,\lambda_q)\in \mathbf{R}^q_+\setminus\{0\}$ such that $SOL(S_{\infty}, \{\sum_{i=1}^{q}\lambda_i f_i\}^{\infty}_d)=\{0\}$.
\end{corollary}

By the definitions of the relative regularity conditions, we  know that the relatively strong regularity implies the relatively weak regularity. The following result gives a their equivalency.

\begin{proposition}\label{2410144}
	Let the nonempty set $S\subseteq \mathbf{R^n}$ with $S_\infty\subseteq K_\infty$. Assume that $f$ is bounded from below on $S$. Then $f$ is relatively strongly regular with $S$ on $K$ if and only if $f$ is relatively weakly regular with $S$ on $K$.
\end{proposition}
\begin{proof}
	Since $f$ is bounded from below on $S$, we have $0\in SOL^{s}(S_{\infty},f^\infty_\mathbf{d})\subseteq SOL^{w}(S_{\infty},f^\infty_\mathbf{d})$ by Proposition \ref{notbounded}. So $f$ is relatively weakly regular with $S$ on $K$ if and only if $SOL^{s}(S_{\infty},f^\infty_\mathbf{d})=\{0\}$. And $f$ is relatively strongly regular with $S$ on $K$ if and only if $SOL^{w}(S_{\infty},f^\infty_\mathbf{d})=\{0\}$. Thus, the result follows from Theorem \ref{2410122}. \end{proof}

In particular, let $\bar x\in K$ and $S=K_{\bar x}$. By Proposition \ref{2410144}, we can obtain the following result.

\begin{corollary}\label{1022}
	Let $\bar x\in K$. If $f$ is section-bounded from below at $\bar x$. Then $f$ is relatively strongly regular with $K_{\bar x}$ on $K$ if and only if $f$ is relatively weakly regular with $K_{\bar x}$ on $K$.
\end{corollary}

\begin{remark}
	The following example shows that the boundedness from below of $f$ in Proposition \ref{2410144} and Corollary \ref{1022} plays an essential role and it cannot be dropped.
	\begin{example}
		Consider the vector polynomial $f=(f_{1}, f_{2})$  with
$$f_{1}(x_{1}, x_{2})=x_{1}, f_{2}(x_{1}, x_{2})=x_{2}$$
and
$$K=\{(x_1,x_2)\in\mathbf{R}^2: x_1\geq 0, x_2\leq 0\}.$$
Clearly, $(f_{1})^{\infty}_{d_1}(x_{1}, x_{2})=x_{1}$ and $(f_{2})^{\infty}_{d_2}(x_{1}, x_{2})=x_{2}$. Let $\bar x=(\bar{x}_1, \bar{x}_2)\in K$ and $S=K_{\bar x}$. It is easy to verify that $S_{\infty}=\{(x_{1}, x_{2})\in \mathbf{R}^{2}: x_{1}=0, x_{2}\leq 0\}$.  Then $SOL^{s}(S_{\infty}, f^{\infty}_\mathbf{d})=\emptyset$ and $SOL^{w}(S_{\infty}, f^{\infty}_\mathbf{d})=\{(x_{1}, x_{2})\in \mathbf{R}^{2}: x_{1}=0, x_{2}\leq 0\}$ is a unbounded set. Thus, we know that $f$ is relatively weakly regular with $S$ on $K$, but $f$ is not relatively strongly regular with $S$ on $K$. On the other hand, it is easy to see that $f$ is not bounded from below on $S$.
	\end{example}
\end{remark}

The following conclusions represent some necessary conditions of the existence of Pareto efficient solutions for the polynomial vector optimizations.

\begin{proposition}\label{2410143}
	The following results hold:
	\begin{itemize}
		\item[(i)] If $SOL^{s}(K, f)\neq\emptyset$, then there exist $\bar x\in K$ and a nonempty closed set $S\subseteq\mathbf{R}^n$ satisfying with $(K_{\bar x})_\infty\subseteq S_\infty\subseteq K_\infty$  such that $SOL(S_{\infty}, \{\sum_{i=1}^{q}\lambda_i f_i\}^{\infty}_d)\neq\emptyset$ for any $\lambda=(\lambda_1,\lambda_2,\dots,\lambda_q)\in\mathbf{R}^q_+\backslash\{0\}$.
		\item[(ii)] If $SOL^{w}(K, f)\neq\emptyset$, then there exist $\bar x\in K$, a nonempty closed set $S\subseteq\mathbf{R}^n$ satisfying with $(K_{\bar x})_\infty\subseteq S_\infty\subseteq K_\infty$, and $\lambda=(\lambda_1,\lambda_2,\dots,\lambda_q)\in \mathbf{R}^q_+\backslash\{0\}$ such that $SOL(S_{\infty}, \{\sum_{i=1}^{q}\lambda_i f_i\}^{\infty}_d)\neq\emptyset$.
	\end{itemize}
\end{proposition}
\begin{proof}
\emph{(i)} By the assumptions, let $\bar x\in SOL^{s}(K, f)$. By \cite[Proposition 3.2]{24LDY3}, we get $f_i(x)\equiv f_i(\bar x)$ for all $i\in \{1,2,\dots,q\}$ and $x\in K_{\bar x}$. Let $S=K_{\bar x}$. Then $(K_{\bar x})_\infty\subseteq S_\infty\subseteq K_\infty$. Let $x\in S_{\infty}$ be arbitrary. Then there exist sequences $\{x_k\}\subseteq S$ and $\{\lambda_k\}$ with $\lambda_k\to +\infty$ as $k\to +\infty$ such that $\frac{x_k}{\lambda_k}\to x$  as $k\to +\infty$. Since $\{x_k\}\subseteq S$, we have $f_i(x_k)\equiv f_i(\bar x)$ for all $i\in \{1,2,\dots,q\}$.  Dividing the both sides of these equalities by $\lambda^{d_i}_{k}$ and then letting $k\rightarrow +\infty$, we get
$$(f_i)^\infty_{d_i}(x)\equiv 0,$$
for each $i=1,2,\dots,q$. Let $\lambda=(\lambda_1,\lambda_2,\dots,\lambda_q)\in \mathbf{R}^q_+\setminus\{0\}$ be arbitrary. Then, by the
arbitrariness of $x$ in $S_{\infty}$, we have $\{\sum_{i=1}^{q}\lambda_i f_i\}^{\infty}_d=\sum_{i\in I}\lambda_i (f_i)^{\infty}_{d_i}\equiv 0$ on $S_{\infty}$, where $I=\{i\in\{1,2,\dots,q\}|deg f_i=\max_{j\in\{1,2,\dots,q\}}\{deg f_j\}\}$. Thus, $SOL(S_{\infty}, \sum_{i=1}^{q}\lambda_i (f_i)^{\infty}_{d_i})\neq\emptyset$.

\emph{(ii)} Since $SOL^{w}(K, f)\neq\emptyset$, by \cite[Proposition 3.1]{24LDY3}, there exist $\bar x\in K$ and $i_0\in \{1,2,\dots,q\}$ such that  $f_{i_0}(x)\equiv f_{i_0}(\bar x)$ for any $x\in K_{\bar x}$. Let $S=K_{\bar x}$. Then $(K_{\bar x})_\infty\subseteq S_\infty\subseteq K_\infty$. Let $x\in S_{\infty}$ be arbitrary. Then there exist sequences $\{x_k\}\subseteq K_{\bar x}$ and $\{\lambda_k\}$ with $\lambda_k\to +\infty$ as $k\to +\infty$ such that $\frac{x_k}{\lambda_k}\to x$  as $k\to +\infty$. Since $\{x_k\}\subseteq S$, we have $f_{i_0}(x_k)\equiv f_{i_0}(\bar x)$ for all $k$.  Dividing the both sides of these equalities by $\lambda^{d_{i_0}}_{k}$ with $d_{i_0}=deg f_{i_0}$ and then letting $k\rightarrow +\infty$, we get
$$(f_{i_0})^\infty_{d_{i_0}}(x)\equiv 0.$$
Thus, by the arbitration of $x$ in $S_{\infty}$, we can prove $SOL(S_{\infty}, \{\sum_{i=1}^{q}\lambda_i f_i\}^{\infty}_d)\neq\emptyset$, where $\lambda_{i_0}=1$, $\lambda_i=0$ with $i\in \{1,2,\dots,q\}\setminus\{i_0\}$ and $d=d_{i_0}$.\end{proof}

For some $\bar x\in K$ and nonempty closed set $S\subseteq\mathbf{R}^n$ satisfying with $(K_{\bar x})_\infty\subseteq S_\infty\subseteq K_\infty$, the following result shows that  $SOL^{s}(S_{\infty}, f^{\infty}_\mathbf{d})\neq\emptyset$ and $SOL^{w}(S_{\infty}, f^{\infty}_\mathbf{d})\neq\emptyset$ are also necessary conditions for the existence of Pareto efficient solutions and weakly Pareto efficient solutions for polynomial vector optimization problems, respectively.

\begin{proposition} \label{1019}
 The following results hold:
  \begin{itemize}
  	\item [(i)] If $SOL^{s}(K, f)\neq\emptyset$, then there exist $\bar x\in K$ and a nonempty closed set $S\subseteq\mathbf{R}^n$ satisfying with $(K_{\bar x})_\infty\subseteq S_\infty\subseteq K_\infty$ such that $SOL^{s}(S_{\infty}, f^{\infty}_\mathbf{d})\neq\emptyset$;
  	\item [(ii)] If $SOL^{w}(K, f)\neq\emptyset$, then there exist $\bar x\in K$ and a nonempty closed set $S\subseteq\mathbf{R}^n$ satisfying with $(K_{\bar x})_\infty\subseteq S_\infty\subseteq K_\infty$ such that $SOL^{w}(S_{\infty}, f^{\infty}_\mathbf{d})\neq\emptyset$.
  \end{itemize}
\end{proposition}
\begin{proof}
	\emph{(i)} Assume that $SOL^{s}(K, f)\neq\emptyset$. Let $\bar x\in SOL^{s}(K, f)$.  Then we see that $f$ is section-bounded from below at $\bar x$. Let $S=K_{\bar x}$. Then $(K_{\bar x})_\infty\subseteq S_\infty\subseteq K_\infty$.  By Corollary \ref{weakpro} \emph{(i)}, we have $0\in SOL^{s}(S_{\infty}, f^{\infty}_\mathbf{d})$. Thus,  $SOL^{s}(S_{\infty}, f^{\infty}_\mathbf{d})\neq\emptyset$.
	
	\emph{(ii)} Assume that $SOL^{w}(K, f)\neq\emptyset$. By \cite[Proposition 3.1]{24LDY3}, we obtain that there exist $\bar x\in K$ and the nonempty index set $I\subseteq\{1,2,\dots,q\}$ such that $f$ is $I$-section-bounded from below at $\bar x$. Let $S=K_{\bar x}$. Then $(K_{\bar x})_\infty\subseteq S_\infty\subseteq K_\infty$. By Corollary \ref{weakpro} \emph{(ii)}, we have $0\in SOL^{w}(S_{\infty}, f^{\infty}_\mathbf{d})$. Thus, $SOL^{w}(S_{\infty}, f^{\infty}_\mathbf{d})\neq\emptyset$.\end{proof}

\begin{remark}
The following example shows that the converse of Propositions \ref{2410143} and \ref{1019} does not hold in general.
\end{remark}

\begin{example}
Consider the vector polynomial  $f=(f_{1}, f_{2})$ with
$$f_{1}(x_{1}, x_{2})=(x^{4}_{1}x^{4}_{2}-1)^{2}+2x^{4}_{1}, f_{2}(x_{1}, x_{2})=(x^{2}_{1}x^{2}_{2}-1)^{2}+4x^{2}_{1}$$
and $K=\mathbf{R}^{2}$. Then $f^\infty_\mathbf{d}(x_1,x_2)=(x^8_1x^8_2, x^4_1x^4_2)$. Let $\bar x\in K$ and $S=K_{\bar x}$. Then $S_\infty=(K_{\bar x})_\infty$. Clearly, $x^8_1x^8_2\geq 0, x^4_1x^4_2\geq 0$ for any $x=(x_1, x_2)\in S_\infty$. It follows from $0\in S_\infty$ that $0\in SOL^{s}(S_{\infty}, f^{\infty}_\mathbf{d})\subseteq SOL^{w}(S_{\infty}, f^{\infty}_\mathbf{d})$ and $(0, 0)\in SOL(S_{\infty}, \{\sum_{i=1}^{2}\lambda_i f_i\}^{\infty}_d)\neq\emptyset$ for any $\lambda=(\lambda_1,\lambda_2)\in\mathbf{R}^2_+\backslash\{(0, 0)\}$. On the other hand, $f_{1}>0$ and $f_{2}>0$ on $K$. However,  $f(\frac{1}{n}, n)=(\frac{2}{n^{4}}, \frac{4}{n^{2}})\to (0, 0)$ as $n\to +\infty$. This implies $SOL^{s}(K, f)\subseteq SOL^{w}(K, f)=\emptyset$.
\end{example}

\section{Relationships between the relative regularity conditions, Palais-Smale condition, weak Palais-Smale condition, M-tameness and properness}

In this section, we investigate relationships between the relative regularity conditions, Palais-Smale condition, weak Palais-Smale condition, M-tameness and properness condition with respect to some index set. First, for nonempty index set $I\subseteq\{1,2,\dots,q\}$,  we recall the definitions of $I$-Palais-Smale condition, $I$-M-tameness and $I$-properness of the restricted mapping $f\mid_K$ of $f$ on $K$.

\begin{definition}\cite[Definition 4.1]{24LDY3}
	Let $I=\{s_1, s_2,\dots, s_p\}\subseteq\{1,2,\dots,q\}$ be a nonempty index set and $f_I: \mathbf{R}^n\to \mathbf{R}^q, f_I=(f_{s_1}, f_{s_2},\dots, f_{s_p})$.
	\begin{itemize}
		\item [(i)] The restricted mapping $f\mid_K$ of $f$ on $K$ is said to be $I$-proper at the sublevel $\bar y\in \mathbf{R}^q$, if $$\forall \{x_k\}\subseteq K,\|x_k\|\to+\infty, f(x_k)\leq \bar y\Longrightarrow \|f_I(x_k)\|\to+\infty\ as\ k\to +\infty;$$
		\item [(ii)] The restricted mapping $f\mid_K$ of $f$ on $K$ is said to be $I$-proper, if it is $I$-proper at every sublevel $\bar y\in \mathbf{R}^q$.
	\end{itemize}
\end{definition}

\begin{remark}
	 As similar to Remark 4.1 in \cite{24LDY3}, when $q=1$ and $f$ is bounded from below, the $I$-properness of the restricted mapping $f|_K$ is equivalent to the coercivity of $f|_K$. When $q\geq 2$, we  know that the $I$-properness of the restricted mapping $f|_K$ is weaker than $\mathbf{R}^q_+$-zero-coercivity of $f$ on $K$ (see e.g. \cite{24LDY3}).
\end{remark}

When $I=\{1,2,\dots,q\}$, we have the following definition.
\begin{definition}\cite[Definition 3.2]{Kim}
	We say that
	\begin{itemize}
		\item [(i)] The restricted mapping $f\mid_K$ of $f$ on $K$ is proper at the sublevel $\bar y\in \mathbf{R}^q$, if $$\forall \{x_k\}\subseteq K,\|x_k\|\to+\infty, f(x_k)\leq \bar y\Longrightarrow \|f(x_k)\|\to+\infty\ as\ k\to +\infty;$$
		\item [(ii)] The restricted mapping $f\mid_K$ of $f$ on $K$ is proper, if it is proper at every sublevel $\bar y\in \mathbf{R}^q$.
	\end{itemize}
\end{definition}

\begin{definition}\cite[Definition 4.2]{24LDY3}\label{IKT}
	  For any nonempty index set $I\subseteq \{1,2,\dots,q\}$ and $y_0\in (\mathbf{R}\cup\{\infty\})^q$, define the following sets:
\begin{equation*}
\begin{split}
&\widetilde{K}^{I}_{\infty,\leq y_0}(f, K):=\{y\in\mathbf{R}^{|I|}\vert \exists \{x_k\}\subseteq K, f(x_k)\leq y_0, \|x_k\|\to +\infty, f_{I}(x_k) \to y \mbox{ and } \nu(x_k)\to 0 \\
&\mbox{ as } k\to +\infty\},\\
&K^{I}_{\infty,\leq y_0}(f, K):=\{y\in\mathbf{R}^{|I|}\vert \exists \{x_k\}\subseteq K, f(x_k)\leq y_0, \|x_k\|\to +\infty, f_{I}(x_k) \to y  \mbox{ and } \|x_k\|\nu(x_k)\\
&\to 0 \mbox{ as } k\to +\infty\},\\
&\mbox{ and }T^{I}_{\infty,\leq y_0}(f, K):=\{y\in\mathbf{R}^{|I|}\vert \exists \{x_k\}\subseteq \Gamma(f, K), f(x_k)\leq y_0, \|x_k\|\to +\infty \mbox{ and } f_I(x_k) \to y \mbox{ as }\\
&k\to +\infty\},
\end{split}
\end{equation*}
where $\nu:\mathbf{R}^n\to \mathbf{R}\cup\{+\infty\}$ is the extended Rabier function defined by
$$\nu(x):=\inf\{\|\sum_{i=1}^{q}\alpha_i \nabla f_i(x)+\omega\|\vert  \omega\in N(x;K), \alpha=(\alpha_1,\alpha_2,\dots,\alpha_q)\in \mathbf{R}^q_+, \sum_{i=1}^{q}\alpha_i=1\},$$
and the tangency variety of $f$ on $K$ defined by
\begin{equation*}
\begin{split}
	&\Gamma(f, K):=\{x\in K\vert \exists (\alpha, \mu)\in \mathbf{R}^q_+\times \mathbf{R} \mbox{ with }  \sum_{i=1}^{q}\alpha_i+|\mu|=1 \mbox{ such that } \\
		&0\in  \sum_{i=1}^{q}\alpha_i \nabla f_i(x)+\mu x+N(x; K)\}.
\end{split}
		\end{equation*}
\end{definition}

\begin{remark}
	By Remark 4.3 in \cite{24LDY3}, we know that the inclusion  $K^I_{\infty, \leq y_0}(f, K)\subseteq \widetilde{K}^I_{\infty, \leq y_0}(f, K)$ holds. When $K=\mathbf{R}^n$,  the inclusion  $T^I_{\infty, \leq y_0}(f, K)\subseteq K^I_{\infty, \leq y_0}(f, K)$ holds. And when $K$ is a closed semi-algebraic set satisfying regularity at infinity, then the inclusion  $T^I_{\infty, \leq y_0}(f, K)\subseteq K^I_{\infty, \leq y_0}(f, K)$ also holds. However, by Remark 4.3 in \cite{24LDY3} again, it is worth noting that if $f$ is not polynomial, then the inclusion  $T^I_{\infty, \leq y_0}(f, K)\subseteq K^I_{\infty, \leq y_0}(f, K)$ may not hold.
\end{remark}

\begin{remark}
	In particular, if $I=\{1,2,\dots,q\}$, then $\widetilde{K}^{I}_{\infty,\leq y_0}(f, K)$, $K^{I}_{\infty,\leq y_0}(f, K)$, and $\mbox{and} \ T^{I}_{\infty,\leq y_0}(f, K)$
 reduce to the following sets (see, e.g., \cite{Kim}):
\begin{equation*}
\begin{split}
&\widetilde{K}_{\infty,\leq y_0}(f, K):=\{y\in\mathbf{R}^s\vert \exists \{x_k\}\subseteq K, f(x_k)\leq y_0, \|x_k\|\to +\infty, f(x_k) \to y, \mbox{ and } \nu(x_k)\to 0 \\
&\mbox{ as } k\to +\infty\},\\
&K_{\infty,\leq y_0}(f, K):=\{y\in\mathbf{R}^s\vert \exists \{x_k\}\subseteq K, f(x_k)\leq y_0, \|x_k\|\to +\infty, f(x_k) \to y, \mbox{ and } \|x_k\|\nu(x_k)\\
&\to 0 \mbox{ as } k\to +\infty\},\\
&\mbox{ and }T_{\infty,\leq y_0}(f, K):=\{y\in\mathbf{R}^s\vert \exists \{x_k\}\subseteq \Gamma(f, K), f(x_k)\leq y_0, \|x_k\|\to +\infty, \mbox{ and }f (x_k) \to y \mbox{ as }\\
& k\to +\infty\}.
\end{split}
\end{equation*}
\end{remark}

\begin{definition}\cite[Definition 4.3]{24LDY3}\label{IPPss}
	Let $I\subseteq \{1,2,\dots,q\}$ be a nonempty index set and $y_0\in (\mathbf{R}\cup\{\infty\})^q$. We say that
	\begin{itemize}
		\item[\rm(i)] $f|_{K}$ satisfies the $I$-Palais-Smale condition at the sublevel $y_0$ if
		$$\widetilde{K}^{I}_{\infty,\leq y_0}(f, K)=\emptyset.$$
		\item[\rm(ii)] $f|_{K}$ satisfies the weak $I$-Palais-Smale condition at the sublevel $y_0$ if
		$$K^{I}_{\infty,\leq y_0}(f, K)=\emptyset.$$
		\item[\rm(iii)] $f|_{K}$ satisfies the $I$-M-tame at the sublevel $y_0$ if
		$$T^{I}_{\infty,\leq y_0}(f, K)=\emptyset.$$
	\end{itemize}
\end{definition}

\begin{remark}\label{0811PS}
	In particular, if $I=\{1,2,\dots,q\}$, then (i)-(iii) of Definition \ref{IPPss} reduce to the Palais-Smale,  weak Palais-Smale  and M-tame condition, that is, $\widetilde{K}_{\infty,\leq y_0}(f, K)=\emptyset$, $K_{\infty,\leq y_0}(f, K)=\emptyset$ and $T_{\infty,\leq y_0}(f, K)=\emptyset$, (see  \cite[Definition 3.3]{Kim}). From the definitions, the properness of the restricted mapping $f|_K$ of $f$ on $K$ with respect to $I$ at sublevel $y_0\in \mathbf{R}^q$ yields
	$\widetilde{K}^{I}_{\infty,\leq y_0}(f, K)=K^{I}_{\infty,\leq y_0}(f, K)=T^{I}_{\infty,\leq y_0}(f, K)=\emptyset$.
The converse does not hold in general, see, e.g. \cite{24LDY3}.
\end{remark}

First, we give the relationships between the relative regularity conditions and properness with respect to some index set as follows.

\begin{theorem}\label{2410233}
Let $\bar x\in K$ and the nonempty index set $I\subseteq\{1,2,\dots,q\}$. Assume that $f$ is $I$-section-bounded from below at $\bar x$. Then the following results are equivalent:
	\begin{itemize}
    \item[(i)] The restricted mapping $f\mid_K$ of $f$ on $K$ is $I$-proper  at the sublevel $f(\bar x)$ ;
    \item[(ii)] There exist a closed set $S\subseteq\mathbf{R}^n$ satisfying with $(K_{\bar x})_\infty\subseteq S_\infty\subseteq K_\infty$ and $\lambda=(\lambda_1,\lambda_2,\dots,\lambda_q)\in \mathbf{R}^q_+\backslash\{0\}$ such that $SOL(S_\infty, \{\sum_{i=1}^{q}\lambda_i f_i\}^{\infty}_d)=\{0\}$;
    \item[(iii)]  There exists a closed set $S\subseteq\mathbf{R}^n$ satisfying with $(K_{\bar x})_\infty\subseteq S_\infty\subseteq K_\infty$ such that $SOL^{w}(S_{\infty}, f^{\infty}_\mathbf{d})=\{0\}$;
    \item[(iv)]  There exists a closed set $S\subseteq\mathbf{R}^n$ satisfying with $(K_{\bar x})_\infty\subseteq S_\infty\subseteq K_\infty$ such that $SOL^{s}(S_{\infty}, f^{\infty}_\mathbf{d})=\{0\}$.
    \end{itemize}
\end{theorem}
\begin{proof}
Let $\lambda=(\lambda_1,\lambda_2,\dots,\lambda_q)\in \mathbf{R}^q_+\backslash\{0\}$ with $\lambda_i\neq 0, i\in I$ and $\lambda_j=0, j\notin I$. Since $f$ is $I$-section-bounded from below at $\bar x$,  the polynomial $\sum_{i=1}^{q}\lambda_i f_i$ is bounded from below on $K_{\bar x}$. So, by Proposition \ref{notbounded}, we have $SOL((K_{\bar x})_\infty, \{\sum_{i=1}^{q}\lambda_i f_i\}^{\infty}_d)\neq\emptyset$.

``\emph{(i)} $\Rightarrow$ \emph{(ii)}": Let $S=K_{\bar x}$. We only need prove $S_\infty=(K_{\bar x})_\infty=\{0\}$. Thus, we assert that $K_{\bar x}$ is bounded. Suppose on the contrary that there exists a sequence $\{x_k\}\subseteq K_{\bar x}$ such that $\|x_k\|\to+\infty$ as $k\to +\infty$. Since $f$ on $K$ is $I$-proper at the sublevel $f(\bar x)$, we have $\|f_I(x_k)\|\to+\infty$ as $k\to +\infty$. It follows from $f(x_k)\leq f(\bar x)$ that there exists $i_0\in I$ such that $f_{i_0}(x_k)\to -\infty$ as $k\to +\infty$. Since $f_{i}$ is bounded from below on $K_{\bar x}$ for any $i\in I$,  there exists $c_{i_0}\in \mathbf{R}$ such that $f_{i_0}(x)\geq c_{i_0}$ for all $x\in K_{\bar x}$. Thus, $f_{i_0}(x_k)\geq c_{i_0}$ for all $k$, which is a contradiction. Thus, $SOL(S_\infty, \{\sum_{i=1}^{q}\lambda_i f_i\}^{\infty}_d)=\{0\}$.

``\emph{(i)} $\Leftarrow$ \emph{(ii)}": Suppose on the contrary that the restricted mapping $f\mid_K$ of $f$ on $K$ is not proper with respect to $I$ at the sublevel $f(\bar x)$. Then there exist $y_0\in \mathbf{R}$ and the sequence $\{y_k\}\subseteq K$ satisfying with $f(y_k)\leq f(\bar x)$ and $\|y_k\|\to +\infty$ as $k\to +\infty$ such that $\|f_I(y_k)\|\leq y_0$ for all $k$. It follows that $|f_i(y_k)|\leq y_0$ for each $i\in I$ and all $k$. Without loss of generality, we assume that $\|y_{k}\|\neq 0$ and $\frac{y_{k}}{\|y_{k}\|}\rightarrow v_0\in (K_{\bar x})_\infty$. Since there exists a closed set $S\subseteq\mathbf{R}^n$ such that $K_{\bar x}\subseteq S$, we have $(K_{\bar x})_\infty\subseteq S_{\infty}$. Thus, $v_0\in S_{\infty}\backslash\{0\}$.  Since for any $i\in I$,
 $$0=\lim_{k\to +\infty}\frac{-y_0}{\|y_k\|^{d_i}}\leq\lim_{k\to +\infty}\frac{f_i(y_k)}{\|y_k\|^{d_i}}=(f_{i})^{\infty}_{d_i}(v)\leq \lim_{k\to +\infty}\frac{y_0}{\|y_k\|^{d_i}}=0,$$
	we have $(f_{i})^{\infty}_{d_i}(v_0)=0$ for each $i\in I$. Thus, by $\{\sum_{i=1}^{q}\lambda_i f_i\}^{\infty}_d(v)\geq 0$ for all $v\in S_\infty$ and $\{\sum_{i=1}^{q}\lambda_i f_i\}^{\infty}_d(0)=0$, we have  $v_0\in SOL(S_\infty, \{\sum_{i=1}^{q}\lambda_i f_i\}^{\infty}_d)\setminus\{0\}$, which is a contradiction.

Finally, by Theorem \ref{2410122}, we know that ``\emph{(ii)} $\Leftrightarrow$ \emph{(iii)}$\Leftrightarrow$ \emph{(iv)}", directly. \end{proof}

By Remark \ref{20250227a}, when the index set $I=\{1,2,\dots,q\}$, we know that the $I$-properness at the sublevel $\bar y\in \mathbf{R}^q$ of the restricted mapping $f\mid_K$ of $f$ on $K$ reduces to the properness at the sublevel $\bar y\in \mathbf{R}^q$, and the $I$-section-boundedness from below reduces to the section-boundedness from below. Thus, by the proof of Theorem \ref{2410233}, we have the following result.

\begin{corollary}\label{2410241}
Let $\bar x\in K$. Assume that $f$ is section-bounded from below at $\bar x$. Then the following results are equivalent:
	\begin{itemize}
    \item[(i)] The restricted mapping $f\mid_K$ of $f$ on $K$ is proper at the sublevel $f(\bar x)$ ;
	\item[(ii)] There exist a closed set $S\subseteq\mathbf{R}^n$ satisfying with $(K_{\bar x})_\infty\subseteq S_\infty\subseteq K_\infty$ and $\lambda=(\lambda_1,\lambda_2,\dots,\lambda_q)\in {\rm{ int }} \mathbf{R}^q_+$ such that $SOL(S_\infty, \{\sum_{i=1}^{q}\lambda_i f_i\}^{\infty}_d)=\{0\}$;
    \item[(iii)]  There exists a closed set $S\subseteq\mathbf{R}^n$ satisfying with $(K_{\bar x})_\infty\subseteq S_\infty\subseteq K_\infty$ such that $SOL^{w}(S_{\infty}, f^{\infty}_\mathbf{d})=\{0\}$;
    \item[(iv)]  There exists a closed set $S\subseteq\mathbf{R}^n$ satisfying with $(K_{\bar x})_\infty\subseteq S_\infty\subseteq K_\infty$ such that $SOL^{s}(S_{\infty}, f^{\infty}_\mathbf{d})=\{0\}$.
\end{itemize}
\end{corollary}

\begin{remark}\label{250310a}
As shown in the proof of  Theorem \ref{2410233}, each of the conclusions (ii)-(iv) in Theorem \ref{2410233} is equivalent to one of the following regularity conditions: the relative $I$-$\mathbf{R}^s_+$-zero-regularity,  relatively weak regularity and relatively strong regularity. Thus, we have the following result.
\end{remark}

\begin{corollary}\label{250310b}
Let $\bar x\in K$. Assume that $f$ is $I$-section-bounded from below at $\bar x$. Then there exists a closed set $S\subseteq\mathbf{R}^n$ satisfying with $(K_{\bar x})_\infty\subseteq S_\infty\subseteq K_\infty$ such that the following results are equivalent:
	\begin{itemize}
    \item[(i)] The restricted mapping $f\mid_K$ of $f$ on $K$ is $I$-proper at the sublevel $f(\bar x)$ ;
    \item[(ii)] $f$ is relatively $I$-$\mathbf{R}^q_+$-zero-regular with $S$  on $K$;
	\item[(iii)] $f$ is relatively strongly regular with $S$ on $K$;
    \item[(iv)] $f$ is relatively weakly regular with $S$ on $K$.
    \end{itemize}
\end{corollary}

\begin{remark}
When $f$ is a convex polynomial mapping on $K$ and $K$ is a convex set, by the proof of  Theorem \ref{2410233} and Corollary  \ref{250310b}, we have the following result.
\end{remark}


\begin{corollary}\label{2410232}
Assume that $f$ is a convex polynomial mapping on $K$ and $K$ is a convex set. Let $\bar x\in K$ and $S=K_\infty\bigcap\{x\in \mathbf{R}^n\mid f(x)\leq f(\bar x)\}_\infty$. If $f$ is section-bounded from below at $\bar x$, then the following results are equivalent:
	\begin{itemize}
    \item[(i)] The restricted mapping $f\mid_K$ of $f$ on $K$ is proper at the sublevel $f(\bar x)$ ;
    \item[(ii)] $f$ is relatively $\mathbf{R}^q_+$-zero-regular with $S$  on $K$;
	\item[(iii)] $f$ is relatively strongly regular with $S$ on $K$;
    \item[(iv)] $f$ is relatively weakly regular with $S$ on $K$.
    \end{itemize}
\end{corollary}

In what follows, we give the relationships between  the relative regularity conditions, $I$-Palais-Smale condition, weak $I$-Palais-Smale condition, $I$-M-tameness, and $I$-properness condition under the $I$-section-boundedness from below condition.

\begin{corollary}\label{relationship}
    Let $\bar x\in K$ and the nonempty index set $I\subseteq\{1,2,\dots,q\}$. Assume that $f$ is  $I$-section-bounded from below at $\bar x$. Then there exists a closed set $S\subseteq\mathbf{R}^n$ satisfying with $(K_{\bar x})_\infty\subseteq S_\infty\subseteq K_\infty$ such that the following assertions are equivalent:
    \begin{itemize}
    \item[(i)] $f|_{K}$ is $I$-proper at the sublevel $f(\bar x)$;
    \item[(ii)] $f$ is relatively $I$-$\mathbf{R}^q_+$-zero-regular with $S$ on $K$;
	\item[(iii)] $f$ is relatively strongly regular with $S$ on $K$;
    \item[(iv)] $f$ is relatively weakly regular with $S$ on $K$;
	\item[(v)] $f|_{K}$ satisfies the $I$-Palais-Smale condition at the sublevel $f(\bar x)$;
	\item[(vi)] $f|_{K}$ satisfies the weak $I$-Palais-Smale condition  at the sublevel $f(\bar x)$;
    \item[(vii)] $f|_{K}$ satisfies $I$-M-tame condition at the sublevel $f(\bar x)$.
    \end{itemize}
Moreover, the set $\{x\in K|f(x)\leq f(\bar x)\}$ and the section $\lbrack f(K) \rbrack_{f(\bar x)}$ are compact if any of the conditions (i)-(vii) is fulfilled.
\end{corollary}
\begin{proof}
	$\lbrack(iii)\Leftrightarrow (iv) \Leftrightarrow (v)\Leftrightarrow (vi)\rbrack$ follows from \cite[Theorem 4.1]{24LDY3}. $\lbrack (i)\Leftrightarrow (ii)\Leftrightarrow (iii)\rbrack$ follows from by Theorem \ref{2410233}.\end{proof}

When the index set $I=\{1,2,\dots,q\}$, we have the following result by \cite[Theorem 3.1]{Kim} and Corollary \ref{relationship}.

\begin{corollary}\label{relationship1}
    Assume that $f$ is  section-bounded from below on $K$ and $\bar x\in K$. Then there exists a closed set $S\subseteq\mathbf{R}^n$ satisfying with $(K_{\bar x})_\infty\subseteq S_\infty\subseteq K_\infty$ such that the following assertions are equivalent:
    \begin{itemize}
    \item[(i)] $f|_{K}$ is proper at the sublevel $f(\bar x)$;
    \item[(ii)] $f$ is relatively $\mathbf{R}^q_+$-zero-regular with $S$ on $K$;
	\item[(iii)] $f$ is relatively strongly regular with $S$ on $K$;
    \item[(iv)] $f$ is relatively weakly regular with $S$ on $K$;
	\item[(v)] $f|_{K}$ satisfies the Palais-Smale condition at the sublevel $f(\bar x)$;
	\item[(vi)] $f|_{K}$ satisfies the weak Palais-Smale condition  at the sublevel $f(\bar x)$;
    \item[(vii)] $f|_{K}$ satisfies M-tame condition at the sublevel $f(\bar x)$.
    \end{itemize}
Moreover, the set $\{x\in K|f(x)\leq f(\bar x)\}$ and the section $\lbrack f(K) \rbrack_{f(\bar x)}$ are compact if any of the conditions (i)-(vii) is fulfilled.
\end{corollary}

\section{Existence results of efficient solutions for ${\rm{PVOP}} (K, f)$}

In this section, under the relative regularity and non-regularity conditions, we shall study nonemptiness of solution sets of ${\rm{PVOP}} (K, f)$ respectively.

\subsection{Existence for ${\rm{PVOP}} (K, f)$ under the relative regularity conditions}
In this subsection, we investigate the existence of the efficient solutions for  polynomial vector optimization problems on a nonempty closed set under the relative regularity conditions without any convexity and compactness assumptions. First, we obtain equivalent characterizations of the sublevel set as follows.

\begin{proposition}\label{250107a}
Let $\bar x\in K$, $\lambda=(\lambda_1,\lambda_2,\dots,\lambda_q)\in \mathbf{R}^q_+\backslash\{0\}$. Then $K_{\bar x}$ is bounded if and only if there exists a closed set $S\subseteq\mathbf{R}^n$ satisfying with $(K_{\bar x})_\infty\subseteq S_\infty\subseteq K_\infty$ such that the one of the following conditions hold:
\begin{itemize}
\item[(i)] $SOL(S_\infty, \{\sum_{i=1}^{q}\lambda_i f_i\}^{\infty}_d)=\{0\}$;
\item[(ii)] $SOL^{s}(S_\infty, f^{\infty}_\mathbf{d})=\{0\}$;
\item[(iii)] $SOL^{w}(S_\infty, f^{\infty}_\mathbf{d})=\{0\}$.
\end{itemize}
\end{proposition}
\begin{proof}
By Theorem \ref{2410122}, we only prove that $K_{\bar x}$ is bounded if and only if there exists a closed set $S\subseteq\mathbf{R}^n$ satisfying with $(K_{\bar x})_\infty\subseteq S_\infty\subseteq K_\infty$ such that the conclusion \emph{(i)} holds.

``$\Leftarrow$": Since there exist a closed set $S\subseteq\mathbf{R}^n$ satisfying with $(K_{\bar x})_\infty\subseteq S_\infty\subseteq K_\infty$ and $\lambda=(\lambda_1,\lambda_2,\dots,\lambda_q)\in \mathbf{R}^q_+\backslash\{0\}$ such that $SOL(S_\infty, \{\sum_{i=1}^{q}\lambda_i f_i\}^{\infty}_d)=\{0\}$, we have $\{\sum_{i=1}^{q}\lambda_i f_i\}^{\infty}_d(v)\geq \{\sum_{i=1}^{q}\lambda_i f_i\}^{\infty}_d(0)=0$ for all $v\in S_\infty$. Let $I=\{i\in\{1,2,\dots,q\}|\deg f_i=\max_{j\in J}\{deg f_j\}$ where $J=\{j: \lambda_j\neq 0, j\in \{1,2,\dots,q\}\} \}$. Then $\{\sum_{i=1}^{q}\lambda_i f_i\}^{\infty}_d=\sum_{i\in I}\lambda_i (f_i)^\infty_{d_i}$. We prove that the set $K_{\bar x}=\{x\in K: f(x)\leq f(\bar x)\}$ is bounded. Suppose on the contrary that there exists a consequence $\{x_k\}\subseteq K_{\bar x}$ such that $\|x_k\|\to +\infty$ as $k\to +\infty$.  Without loss of generality, we can assume that $\|x_{k}\|\neq 0$ and $\frac{x_{k}}{\|x_{k}\|}\rightarrow v_0$. It follows from $\{x_k\}\subseteq K_{\bar x}$ that $v_0\in (K_{\bar x})_\infty\backslash\{0\}$. Since $\{x_k\}\subseteq K_{\bar x}$, we have $f_i(x_k)\leq f_i(\bar x), i\in \{1,2,\dots,q\}$. Dividing the both sides of these inequalities by $\|x_k\|^{d_{i}}$ with $d_{i}=\deg f_{i}$ and then letting $k\rightarrow +\infty$, we get
\begin{equation}\label{250107c}
(f_i)^\infty_{d_i}(v_0)\leq 0, i\in \{1,2,\dots,q\}.
\end{equation}
Since $\{\sum_{i=1}^{q}\lambda_i f_i\}^{\infty}_d(x_k)=\sum_{i\in I}\lambda_i (f_i)^\infty_{d_i}(x_k)$, we have that
$$\frac{1}{\|x_k\|^d}\{\sum_{i=1}^{q}\lambda_i f_i\}^{\infty}_d(x_k)=\sum_{i\in I}\lambda_i \frac{1}{\|x_k\|^d}(f_i)^\infty_{d_i}(x_k)=\sum_{i\in I}\lambda_i (f_i)^\infty_{d_i}(\frac{x_k}{\|x_k\|})$$
with $d=\deg f_i, i\in I$. This together with inequalities (\ref{250107c}) and let $k\to +\infty$, we have $\{\sum_{i=1}^{q}\lambda_i f_i\}^{\infty}_d(v_0)\leq 0$. And so, $\{\sum_{i=1}^{q}\lambda_i f_i\}^{\infty}_d(v_0)=0$. Since $(K_{\bar x})_\infty\subseteq S_{\infty}$, we have $v_0\in S_{\infty}\backslash\{0\}$. Thus, $v_0\in SOL(S_\infty, \{\sum_{i=1}^{q}\lambda_i f_i\}^{\infty}_d)\setminus\{0\}$, which is a contradiction with $SOL(S_\infty, \{\sum_{i=1}^{q}\lambda_i f_i\}^{\infty}_d)=\{0\}$.

``$\Rightarrow$": It is clearly, since $K_{\bar x}$ is bounded if and only if $(K_{\bar x})_\infty=\{0\}$, we only let $S=K_{\bar x}$.
\end{proof}

\begin{remark}\label{250107d}
Let $\bar x\in K$. By the proof of Proposition \ref{250107a}, we know that the choice of $S$ depends on $K_{\bar x}$ in conditions of Proposition \ref{250107a}. However, since the inclusion $(K_{\bar x})_\infty\subseteq K_\infty\bigcap\{x\in \mathbf{R}^n\mid f(x)\leq f(\bar x)\}_{\infty}\subseteq K_\infty\bigcap\{x\in \mathbf{R}^n\mid f^\infty_\mathbf{d}(x)\leq 0\}$ naturally valid, by the proof of Proposition \ref{250107a}, we know that if a closed set $S\subseteq\mathbf{R}^n$ satisfies with $S\in \{S'\subseteq\mathbf{R}^n\vert K_\infty\bigcap\{x\in \mathbf{R}^n\mid f(x)\leq f(\bar x)\}_{\infty}\subseteq (S')_\infty\subseteq K_\infty\}$, which is independent of $K_{\bar x}$, such that the one of conditions (i),(ii) and (iii) in Proposition \ref{250107a} holds, then we have the following result.
\end{remark}

\begin{corollary}
Let $\bar x\in K$, $\lambda=(\lambda_1,\lambda_2,\dots,\lambda_q)\in \mathbf{R}^q_+\backslash\{0\}$ and a nonempty set $S\subseteq\mathbf{R}^n$ satisfying with $K_\infty\bigcap\{x\in \mathbf{R}^n\mid f(x)\leq f(\bar x)\}_{\infty}\subseteq S_\infty\subseteq K_\infty$. If the one of the following conditions hold:
\begin{itemize}
\item[(i)] $SOL(S_\infty, \{\sum_{i=1}^{q}\lambda_i f_i\}^{\infty}_d)=\{0\}$;
\item[(ii)] $SOL^{s}(S_\infty, f^{\infty}_\mathbf{d})=\{0\}$;
\item[(iii)] $SOL^{w}(S_\infty, f^{\infty}_\mathbf{d})=\{0\}$.
\end{itemize}
Then $K_{\bar x}$ is bounded.
\end{corollary}

Next, we obtain the existence of the Pareto efficient solutions for ${\rm{PVOP}} (K, f)$.

\begin{theorem}\label{exist0}
Let $\bar x\in K$ and $\lambda=(\lambda_1,\lambda_2,\dots,\lambda_q)\in \mathbf{R}^q_+\backslash\{0\}$. If there exists a closed set $S\subseteq\mathbf{R}^n$ satisfying with $(K_{\bar x})_\infty\subseteq S_\infty\subseteq K_\infty$ such that the one of the following conditions hold:
\begin{itemize}
\item[(i)] $SOL(S_\infty, \{\sum_{i=1}^{q}\lambda_i f_i\}^{\infty}_d)=\{0\}$;
\item[(ii)] $SOL^{s}(S_\infty, f^{\infty}_\mathbf{d})=\{0\}$;
\item[(iii)] $SOL^{w}(S_\infty, f^{\infty}_\mathbf{d})=\{0\}$,
\end{itemize} then $SOL^{s}(K, f)$ is nonempty. In addition, if $S$ satisfies with $K_\infty\cap(\bigcup^q_{i=1}\{x\in\mathbf{R}^n\vert (f_i)^{\infty}_{d_i}(x)\leq 0\})\subseteq S_\infty\subseteq K_\infty$ in conclusion (iii), then $SOL^{s}(K, f)$ is also bounded.
\end{theorem}
\begin{proof}
Let $\alpha=(\alpha_1,\alpha_2,\dots,\alpha_q)\in \rm{ int }\mathbf{R}^{q}_{+}$. Define $g_\alpha(x)=\sum^{q}_{i=1}\alpha_i f_{i}(x)$. Then, by Proposition \ref{250107a},  $K_{\bar x}$ is bounded,  and so $K_{\bar x}$ is a compact set. Thus, by $\emph{Weierstrass}'$ Theorem, we have $SOL(K_{\bar x}, g_\alpha)\neq\emptyset$. Since $SOL(K_{\bar x}, g_\alpha)\subseteq SOL^{s}(K, f)$ (by \cite[Proposition 13]{MG}), we have $SOL^{s}(K, f)\neq\emptyset$. In addition, suppose on the contrary that there exists a consequence $\{x_k\}\subseteq SOL^{s}(K, f)$ such that $\|x_k\|\to +\infty$ as $k\to +\infty$.  Without loss of generality, we can assume that $\|x_{k}\|\neq 0$ and $\frac{x_{k}}{\|x_{k}\|}\rightarrow v_0\setminus\{0\}$. Fix any $x_0\in K$. Since $x_k\in SOL^{s}(K, f)$ for any $k$, there exists $i_k\in \{1,2,\dots,q\}$ such that $f_{i_k}(x_0)\geq f_{i_k}(x_k)$ for any $k$. Because the set $\{1,2,\dots,q\}$ is finite, without loss of generality, we suppose that there exists $i_0\in \{1,2,\dots,q\}$ such that $f_{i_0}(x_0)\geq f_{i_0}(x_k)$ for any $k$. Dividing the both side of the above inequality by $\|x_k\|^{d_{i_0}}$ and then letting $k\rightarrow +\infty$, we get
$$(f_{i_0})^{\infty}_{d_{i_0}}(v_0)\leq 0.$$
Then $v_0\in K_\infty\cap(\bigcup^q_{i=1}\{x\in\mathbf{R}^n\vert (f_i)^{\infty}_{d_i}(x)\leq 0\})$. Since $K_\infty\cap(\bigcup^q_{i=1}\{x\in\mathbf{R}^n\vert (f_i)^{\infty}_{d_i}(x)\leq 0\})\subseteq S_\infty\subseteq K_\infty$, we have $v_0\in S_\infty$. This implies $v_0\in SOL^{w}(S_\infty, f^{\infty}_\mathbf{d})\setminus\{0\}$, a contradiction. Thus, $SOL^{s}(K, f)$ is bounded. \end{proof}

\begin{remark}\label{250107e}
In particular, if $S_\infty=K_\infty$ in Theorem \ref{exist0}, then we infer Theorems 5.1 and 5.8 in \cite{LDY1}. Similar to the discussion of Remark \ref{250107d}, we know that if $S\subseteq\mathbf{R}^n$ satisfies with $S\in \{S'\subseteq\mathbf{R}^n\vert K_\infty\bigcap\{x\in \mathbf{R}^n\mid f(x)\leq f(\bar x)\}_{\infty}\subseteq (S')_\infty\subseteq K_\infty\}$ such that the one of conditions (i),(ii) and (iii) in Theorem \ref{exist0} holds, then we have $SOL^{s}(K, f)$ is nonempty.
\end{remark}

Now, we give a following example to illustrate Theorem \ref{exist0}.

\begin{example}\label{1}
	Consider the vector polynomial $f=(f_{1}, f_{2})$ with
 $$f_{1}(x_{1}, x_{2})=x^{3}_{2}-x^{2}_{1}-x_{1}x_{2}+1, f_{2}(x_{1}, x_{2})=x^{2}_{1}-1$$ and
  $$K=\{(x_{1}, x_{2})\in \mathbf{R}^{2}: x_{1}\geq 0, x_{2}\geq 0,e^{x_{1}}-x_{2}\geq 0\}.$$
Then $(f_{1})^{\infty}_{d_1}(x_{1}, x_{2})=x^{3}_{2},(f_{2})^{\infty}_{d_2}(x_{1}, x_{2})=x^{2}_{1}$.
Let $S=K_\infty\bigcap\{x\in \mathbf{R}^2\mid f^\infty_\mathbf{d}(x)\leq 0\}$. Then $S=S_\infty$. It is easy to prove that $S_\infty=\{(0,0)\}$. So $SOL(S_\infty, \{\sum_{i=1}^{2}\lambda_i f_i\}^{\infty}_d)=\{0\}$ for all $ \lambda=(\lambda_1,\lambda_2)\in \mathbf{R}^2_+\backslash\{0\}$. By Theorem \ref{exist0}, we have $SOL^{s}(K, f)\neq\emptyset$. It is worth mentioning that \cite[Theorem 5.1]{LGJ1}, \cite[Theorem 4.1]{DTN}, \cite[ Theorem 3.1]{LGJ}, \cite[ Theorem 3.2, 3.10]{LDY2} and \cite[ Theorem 5.8]{LDY1} cannot be applied in this example since $f$ is non-convex and non-regular on $K$, and $K$ is neither convex nor semi-algebraic set.
\end{example}

The following example shows that if the one of the conditions \emph{(i)}-\emph{(iii)} in Theorem \ref{exist0} holds, then $SOL^{s}(K, f)$ is nonempty. However, $SOL^{s}(K, f)$ may be unbounded.

\begin{example}
		Consider the vector polynomial $f=(f_{1}, f_{2})$ with $f_{1}(x_{1}, x_{2})=2x^{2}_{1}-x_{2}, f_{2}(x_{1}, x_{2})=x^{3}_{2}$ and
$$K=\{(x_{1}, x_{2})\in \mathbf{R}^{2}: x_{2}\geq x_{1}\geq 0\}.$$
Then $(f_{1})^{\infty}_{d_1}(x_{1}, x_{2})=2x^{2}_{1}$ and $(f_{2})^{\infty}_{d_2}(x_{1}, x_{2})=x^{3}_{2}$. Let $S=K_\infty\bigcap\{x\in \mathbf{R}^2\mid f^\infty_\mathbf{d}(x)\leq 0\}$. Then $S_\infty=S$. It is easy to prove that $S_\infty=\{(0,0)\}$, and so, $SOL(S_\infty, \{\sum_{i=1}^{2}\lambda_i f_i\}^{\infty}_d)=\{0\}$ for all $ \lambda=(\lambda_1,\lambda_2)\in \mathbf{R}^2_+\backslash\{0\}$. On the other hand, $SOL^{s}(K, f)$ is unbounded since
$$\{(x_{1}, x_{2})\in K: x_{1}=0, x_{2}\geq 0\}\subseteq SOL^{s}(K, f).$$
	\end{example}

The following example shows that the converse of Theorem \ref{exist0} does not hold in general.

\begin{example}\label{1029}
	Consider the polynomial $f=(f_{1}, f_{2})$ with
$$f_{1}(x_{1}, x_{2})=x_{1}-x_{2},\quad f_{2}(x_{1}, x_{2})=x_{2}-x_{1}$$ and
$$K=\{(x_{1}, x_{2})\in \mathbf{R}^{2}: x_{1}\geq 0, x_{2}\geq 0, x_1=x_2\}.$$
	Then $(f_{1})^{\infty}_{d_1}(x_{1}, x_{2})=x_{1}-x_{2}, (f_{2})^{\infty}_{d_2}(x_{1}, x_{2})=x_{2}-x_{1}$. Let $\bar x\in K$ and the set $S$ satisfying with $K_{\bar x}\subseteq S\subseteq K$ be arbetrary. Then, it is easy to prove $SOL(S_\infty, \{\sum_{i=1}^{2}\lambda_i f_i\}^{\infty}_d)=\{(x_{1}, x_{2})\in \mathbf{R}^{2}: x_{1}=x_{2}\geq 0\}$ for all $ \lambda=(\lambda_1,\lambda_2)\in \mathbf{R}^2_+\backslash\{0\}$, which is unbounded. On the other hand, $SOL^{s}(K, f)=K$.
\end{example}

From Example \ref{1029}, we have known that the inverse of Theorem \ref{exist0} may not hold. However, we have the following result.

\begin{proposition}\label{n1029}
	If $SOL^{s}(K, f)$ is nonempty and bounded, then there exists a closed set $S\subseteq\mathbf{R}^n$ satisfying with $(K_{\bar x})_\infty\subseteq S_\infty\subseteq K_\infty$ such that the following results hold:
\begin{itemize}
\item[(i)] $SOL(S_\infty, \{\sum_{i=1}^{q}\lambda_i f_i\}^{\infty}_d)=\{0\}$ for any $\lambda=(\lambda_1,\lambda_2,\dots,\lambda_q)\in \mathbf{R}^q_+\backslash\{0\}$;
\item[(ii)] $SOL^{s}(S_\infty, f^{\infty}_\mathbf{d})=\{0\}$;
\item[(iii)] $SOL^{w}(S_\infty, f^{\infty}_\mathbf{d})=\{0\}$.
\end{itemize}
\end{proposition}
\begin{proof}
	Let $\bar x\in SOL^{s}(K, f)$ and $S=K_{\bar x}$. Since $S_\infty\subseteq \{x\in \mathbf{R}^n\mid f^\infty(x)\leq 0\}$, the above conclusions \emph{(i)}-\emph{(iii)} are equivalent by Theorem \ref{2410122}. Thus, we only need to prove the conclusion \emph{(i)} holds. We  claim that $K_{\bar x}=\{x\in K: f_i(x)\leq f_i(\bar x), i=1,2,\dots, q\}$ is bounded, since if $K_{\bar x}$ is bounded, then $(K_{\bar x})_\infty=\{0\}$, and so, $SOL(S_\infty, \{\sum_{i=1}^{q}\lambda_i f_i\}^{\infty}_d)=\{0\}$ with $\lambda=(\lambda_1,\lambda_2,\dots,\lambda_q)\in \mathbf{R}^q_+\backslash\{0\}$. Suppose on the contrary that $K_{\bar x}$ is unbounded. Then there exists a sequence  $\{x_{k}\}\subset K_{\bar x}$ such that $\|x_{k}\|\to +\infty$ as $k\to +\infty$. Since $\bar x\in SOL^{s}(K, f)$, we have the section $\lbrack f(K)\rbrack_{f(\bar x)}=\{f(\bar x)\}$. So $f(x_k)=f(\bar x)$ for all $k$. And so $\{x_k\}\subseteq SOL^{s}(K_{\bar x}, f)$. Thus, by \cite[Proposition 3.2]{24LDY3}, we have $\{x_k\}\subseteq SOL^{s}(K, f)$, which is a contradiction with the boundedness of $SOL^{s}(K, f)$.\end{proof}

\begin{remark}
	Example \ref{1029} shows that the boundedness of $SOL^{s}(K, f)$ in Theorem \ref{n1029} plays an essential role and it cannot be dropped.
\end{remark}

The following results give Frank-Wolfe type theorems for  ${\rm{PVOP}} (K, f)$ under the relative regularity conditions.

\begin{corollary}\label{wexist1}[Frank-Wolfe type theorems for ${\rm{PVOP}}(K, f)$]
The following results hold:
\begin{itemize}
	\item [(i)] Assume that there exist $\bar x\in K$ and the some nonempty index set $I\subseteq\{1,2,\dots,q\}$ such that the vector polynomial $f$ is relatively $I$-$\mathbf{R}^q_+$-zero-regular with $K_{\bar x}$ on $K$. If $f$ is $I$-section-bounded from below at $\bar x$, then $SOL^{s}(K, f)$ is nonempty;
	\item [(ii)] Assume that there exists $\bar x\in K$ such that the vector polynomial $f$ is relatively $\mathbf{R}^q_+$-zero-regular with $K_{\bar x}$ on $K$. If $f$ is section-bounded from below at $\bar x$, then $SOL^{s}(K, f)$ is nonempty;
    \item [(iii)] Assume that there exists $\bar x\in K$ such that the vector polynomial $f$ is relatively strongly regular with $K_{\bar x}$ on $K$. If $f$ is  $I$-section-bounded from below at $\bar x$ for the some nonempty index set $I\subseteq\{1,2,\dots,q\}$, then $SOL^{s}(K, f)$ is nonempty;
	\item [(iv)] Assume that there exists $\bar x\in K$ such that the vector polynomial $f$ is relatively weakly regular with $K_{\bar x}$ on $K$. If $f$ is section-bounded from at $\bar x$, then $SOL^{s}(K, f)$ is nonempty.
\end{itemize}
\end{corollary}
\begin{proof}
\emph{(i)} Since $f$ is $I$-section-bounded from below  at $\bar x\in K$, there exists $\lambda=(\lambda_1,\lambda_2,\dots,\lambda_q)\in \mathbf{R}^q_+\backslash\{0\}$ such that $SOL((K_{\bar x})_\infty, \{\sum_{i=1}^{q}\lambda_i f_i\}^{\infty}_d)\neq\emptyset$ by Corollary \ref{2410235}. By the definition of the relative $I$-$\mathbf{R}^q_+$-zero-regularity, we have  $SOL((K_{\bar x})_\infty, \{\sum_{i=1}^{q}\lambda_i f_i\}^{\infty}_d)=\{0\}$. Thus, \emph{(i)} follows from Theorem \ref{exist0}.

 \emph{(ii)} Since $f$ is section-bounded from below at $\bar x\in K$, we have $SOL((K_{\bar x})_\infty, \{\sum_{i=1}^{q}\lambda_i f_i\}^{\infty}_d)\neq\emptyset$ for all $\lambda=(\lambda_1,\lambda_2,\dots,\lambda_q)\in  \rm{ int } \mathbf{R}^q_+$ by Corollary \ref{2410235}. By  the definition of the relative $I$-$\mathbf{R}^q_+$-zero-regularity, we have $SOL((K_{\bar x})_\infty, \{\sum_{i=1}^{q}\lambda_i f_i\}^{\infty}_d)=\{0\}$. Thus, \emph{(ii)} follows from Theorem \ref{exist0}.

 \emph{(iii)} Since $f$ is $I$-section-bounded from below at $\bar x$, by Corollary \ref{weakpro} \emph{(ii)}, we have $SOL^{w}(S_{\infty}, f^{\infty}_\mathbf{d})\neq\emptyset$. It follows from the definition of relatively strong regularity that $SOL^{w}(S_{\infty}, f^{\infty}_\mathbf{d})=\{0\}$. By Theorem \ref{exist0}, we have $SOL^{s}(K, f)$ is nonempty.

 \emph{(iv)} Since $f$ is section-bounded from below at $\bar x$, by Corollary \ref{weakpro} \emph{(i)}, we have  $SOL^{s}(S_{\infty}, f^{\infty}_\mathbf{d})\neq\emptyset$. It follows from the definition of relatively weak regularity that $SOL^{s}(S_{\infty}, f^{\infty}_\mathbf{d})=\{0\}$. By Theorem \ref{exist0}, we have $SOL^{s}(K, f)$ is nonempty.
\end{proof}

Let $\bar x\in K$. When $f$ is a convex polynomial mapping on $K$ and $K$ is a convex set, by Corollary \ref{wexist1}, we have the following result.

\begin{corollary}\label{wexist2}
Assume that $K$ is a convex set and $f$ is a convex polynomial mapping on $K$. Let $S=K_\infty\bigcap\{x\in \mathbf{R}^n\mid f(x)\leq f(\bar x)\}_\infty$ with $\bar x\in K$. The following results hold:
\begin{itemize}
	\item [(i)] Assume that the vector polynomial $f$ is relatively $I$-$\mathbf{R}^q_+$-zero-regular with $S$ on $K$ for the some nonempty index set $I\subseteq\{1,2,\dots,q\}$. If $f$ is $I$-section-bounded from below at $\bar x$, then $SOL^{s}(K, f)$ is nonempty;
	\item [(ii)] Assume that the vector polynomial $f$ is relatively $\mathbf{R}^q_+$-zero-regular with $S$ on $K$. If $f$ is section-bounded from below at $\bar x$, then $SOL^{s}(K, f)$ is nonempty;
    \item [(iii)] Assume that the vector polynomial $f$ is relatively strongly regular with $S$ on $K$. If $f$ is $I$-section-bounded from below  at $\bar x$ for the some nonempty index set $I\subseteq\{1,2,\dots,q\}$, then $SOL^{s}(K, f)$ is nonempty;
	\item [(iv)] Assume that the vector polynomial $f$ is relatively weakly regular with $S$ on $K$. If $f$ is section-bounded from at $\bar x$, then $SOL^{s}(K, f)$ is nonempty.
\end{itemize}
\end{corollary}

When $q=1$, we know that  $I$-section-boundedness from below of $f$ is equivalence to boundedness from below of $f$. Thus,  we have the following result.

\begin{corollary}\label{1028}[Frank-Wolfe type theorem for ${\rm{PSOP}}(K, f)$]The following statements are equivalent:
\begin{itemize}
	\item [(i)] The scalar polynomial $f_1$ satisfies that $f_1$ is bounded from below on $K$ and there exist $\bar x\in K$ and a closed set $S\subseteq\mathbf{R}^n$ satisfying with  $(K_{\bar x})_\infty\subseteq S_\infty\subseteq K_\infty$ such that the scalar polynomial $f_1$ is relatively regular with $S$ on $K$.
	\item [(ii)] $SOL(K, f_1)$ is nonempty and bounded.
\end{itemize}
\end{corollary}
\begin{proof}
	``\emph{(i)} $\Rightarrow$ \emph{(ii)}": Since $f_1$ is bounded from below on $K$,we have $SOL(S_\infty, (f_1)^{\infty}_{d_1})\neq \emptyset$. Thus, by relative regularity with $S$ on $K$ of $f_1$, we have $SOL(S_\infty, (f_1)^\infty_{d_1})=\{0\}$. It follows from $(K_{\bar x})_\infty\subseteq S_\infty\subseteq K_\infty$ that $SOL(K, f_1)\neq\emptyset$ by Theorem \ref{exist0}. Next, we prove that $SOL(K, f_1)$ is bounded. Suppose on the contrary that there exists $\{x_{k}\}\subseteq SOL(K, f_1)$ such that $\|x_{k}\|\rightarrow +\infty$ as $k\rightarrow +\infty$. Without loss of generality, we can assume that $\|x_{k}\|\neq 0$ and $\frac{x_{k}}{\|x_{k}\|}\rightarrow v_0\in K_{\infty}\backslash\{0\}$. Since $x_{k}\in SOL(K, f_1)$ for all $k$, we have $f_1(x_k)\leq f_1(\bar x)$ for all $k$. Thus, $\{x_{k}\}\subseteq K_{\bar x}$. So $v_0\in (K_{\bar x})_{\infty}\backslash\{0\}\subseteq S_{\infty}\backslash\{0\}$ and
	\begin{equation}\label{1102}
		(f_1)^\infty_{d_1}(v_0)=\lim_{k\to+\infty}\frac{f_1(x_k)}{\|x_k\|^{d_1}}\leq \lim_{k\to+\infty}\frac{f_1(\bar x)}{\|x_k\|^{d_1}}=0.
	\end{equation}
	By $SOL(S_\infty, (f_1)^\infty_{d_1})=\{0\}$ and $(f_1)^\infty_{d_1}(0)=0$, we have $(f_1)^\infty_{d_1}\geq 0$ on $S_\infty$. This together with (\ref{1102}) that  $v_0\in SOL(S_\infty, (f_1)^\infty_{d_1})\backslash\{0\}$, which is a contradiction.\end{proof}
	
	``\emph{(ii)} $\Rightarrow$ \emph{(i)}": Since $SOL(K, f_1)$ is nonempty,  $f_1$ is bounded from below on $K$. Applied Proposition \ref{n1029} to the case $q=1$, we know that there exists a closed set $S\subseteq\mathbf{R}^n$ satisfying with  $(K_{\bar x})_\infty\subseteq S_\infty\subseteq K_\infty$ such that $SOL(S_\infty, (f_1)^\infty_{d_1})=\{0\}$. Thus, $f_1$ is relatively regular with $S$ on $K$.

\begin{remark}
It's worth noting that \cite{Hung} used the tangency values at infinity condition to provide necessary and sufficient conditions of the non-emptiness and compactness of the solution set for a scalar optimization problem. However, Corollary \ref{1028} gives a necessary and sufficient condition for a scalar polynomial optimization problem by utilizing the relative regularity condition.
\end{remark}

\begin{remark}
	If $f_1$ is coercive on $K$, then we know that $SOL(K, f_1)$ is nonempty and bounded. If $f_1$ is regular on $K$ and bounded from below on $K$, we have that $SOL(K, f_1)$ is nonempty and bounded, see \cite[Theorem 3.1]{HV1}. The following example shows that the statement (i) of Corollary \ref{1028} is weaker than the coercivity condition and is also  weaker  than the conditions in \cite[Theorem 3.1]{HV1}. Thus, Corollary \ref{1028} extends and improves \cite[Theorem 3.1]{HV1}.
	
	\begin{example}\label{co1029}
    Consider the polynomial $f_1: \mathbf{R}^2\mapsto \mathbf{R}, f_1(x)=x_1x^2_2-x_1x_2$ and
    $$K=\{(x_{1}, x_{2})\in \mathbf{R}^{2}: x_{2}\geq x_{1}\geq 0\}.$$
    Let $\bar x=(\frac{1}{2},\frac{1}{2})$. Then,
    \begin{equation*}
    	\begin{split}
    		K_{\bar x}&=\{(x_{1}, x_{2})\in \mathbf{R}^{2}: x_{2}\geq x_{1}\geq 0, f_1(x)\leq f_1(\bar x)\}\\
    		&=\{(x_{1}, x_{2})\in \mathbf{R}^{2}: x_{2}\geq x_{1}\geq 0, x_1x_2(x_2-1)\leq -\frac{1}{8}\}
    	\end{split}
    \end{equation*}
    It is easy to prove that $f_1$ is bounded from below on $K$. We assert $x_2-1<0$ for any $(x_{1}, x_{2})\in K_{\bar x}$. Otherwise, there exists $(z_1, z_2)\in K_{\bar x}$ such that $z_2-1\geq 0$. Since $z_{2}\geq z_{1}\geq 0$, we have $z_1z_2(z_2-1)\geq 0$, which is a contradiction with $(z_1, z_2)\in K_{\bar x}$. Thus, we have $0\leq x_1\leq x_2<1$ for any $(x_{1}, x_{2})\in K_{\bar x}$. So $K_{\bar x}$ is bounded, and so $(K_{\bar x})_\infty=\{(0,0)\}$. Therefore, $SOL((K_{\bar x})_{\infty}, (f_1)^\infty_{d_1})=\{(0,0)\}$. So $f$ is relatively regular with $S=K_{\bar x}$ on $K$. By Corollary \ref{1028}, we have that $SOL(K, f_1)$ is nonempty and bounded. On the one hand, let $x_n=(\frac{1}{n(n-1)}, n), n\geq 2$. Then $x_n\in K$ and $\|x_n\|\to +\infty$ as $n\to +\infty$. However, $\lim_{n\to +\infty}f(x_n)=1$. Thus, $f$ is not coercive on $K$. On the other hand, $(f_1)^\infty_{d_1}(x)=x_1x^2_2\geq 0$ on $K$ and $K_\infty=K$. So $SOL(K_{\infty}, (f_1)^\infty_{d_1})=\{(x_{1}, x_{2})\in \mathbf{R}^{2}: x_{2}\geq 0, x_{1}=0\}$ is a unbounded set. Thus, $f$ is non-regular on $K$.
 \end{example}
\end{remark}

Finally, we give an application of the existence of Pareto efficient solutions for the polynomial vector optimization problems with the closed constraint set, directly. By Corollaries \ref{relationship} and \ref{wexist1}, we have the following result.

\begin{corollary}
	Assume that there exist $\bar x\in K$ and nonempty index set $I\subseteq\{1,2,\dots,q\}$ such that the vector polynomial $f$ is $I$-section-bounded from below at $\bar x$. Then $\rm{PVOP}$$(K, f)$ admits at least one Pareto efficient solution, if one of the following equivalent conditions holds:
	\begin{itemize}
    \item[(i)] $f|_{K}$ is relatively $I$-$\mathbf{R}^q_+$-zero-regular with $K_{\bar x}$ on $K$;
    \item[(ii)] $f|_{K}$ $f$ is relatively strongly regular with $K_{\bar x}$ on $K$;
    \item[(iii)] $f|_{K}$ $f$ on $K$ is $I$-proper  at the sublevel $f(\bar x)$;
	\item[(iv)] $f|_{K}$ satisfies the $I$-Palais-Smale condition at the sublevel $f(\bar x)$;
	\item[(v)] $f|_{K}$ satisfies the weak $I$-Palais-Smale condition at the sublevel $f(\bar x)$;
    \item[(vi)] $f|_{K}$ satisfies $I$-M-tame condition at the sublevel $f(\bar x)$.
    \end{itemize}
    In particular, if $I=\{1,2,\dots,q\}$ and $f$ is relatively weakly regular with $K_{\bar x}$ on $K$, then the Pareto efficient solution set of $\rm{PVOP}$$(K, f)$ is also nonempty.
\end{corollary}

\subsection{Existence for ${\rm{PVOP}} (K, f)$ under the relative non-regularity conditions}
In this subsection, we investigate the existence of the efficient solutions for polynomial vector optimization problems  on a nonempty closed set  without any convexity and compactness assumptions under the relatively non-regularity conditions.

\begin{theorem}\label{241227c}
If the following conditions hold:
\begin{itemize}
	\item [(i)] For any $x\in K$ and nonempty closed set $S\subseteq\mathbf{R}^n$ satisfying with  $(K_{\bar x})_\infty\subseteq S_\infty\subseteq K_\infty$ such that the set $SOL^{w}(S_\infty, f^{\infty}_\mathbf{d})$ is unbounded, this is, the vector polynomial $f$ is relatively strongly non-regular on $K$.
	\item [(ii)] And for every $v\in SOL^{w}(S_\infty, f^{\infty}_\mathbf{d})\setminus\{0\}$, there exists $t>0$ such that $x-tv\in K$ and $f(x-tv)\leq f(x)$ for all $x\in S$.
\end{itemize}
Then $SOL^{s}(K, f)$ is nonempty.
\end{theorem}
\begin{proof}
Let $\bar x\in K$ and $S=K_{\bar x}$. Then $(K_{\bar x})_\infty\subseteq S_\infty\subseteq K_\infty$. For all sufficiently large $k$, we can know that $S\cap k\mathbf{B}\neq\emptyset$. Let $\lambda=(\lambda_1,\lambda_2,\dots,\lambda_q)\in {\rm{ int }} \mathbf{R}^q_+$. Consider the following optimization problems:
$${\rm{POP}}(S\cap k\mathbf{B}, \sum^{q}_{i=1}\lambda_i f_i): \min_{x\in S\cap k\mathbf{B}}\sum^{q}_{i=1}\lambda_i f_i(x).$$
Clearly, $S\cap k\mathbf{B}$ is compact.  According to Weierstrass' Theorem, ${\rm{POP}}(S\cap k\mathbf{B}, \sum^{q}_{i=1}\lambda_i f_i)$ has a solution. We set
\begin{equation}\label{241226b}
	\|x_k\|=\min\{x\vert x\in SOL(S\cap k\mathbf{B}, \sum^{q}_{i=1}\lambda_i f_i)\}.	
\end{equation}
We claim that $\{x_k\}$ is bounded. Supposed on the contrary that $\|x_{k}\|\rightarrow +\infty$ as $k\rightarrow +\infty$. Without loss of generality, we can assume that $\|x_{k}\|\neq 0$ and $\frac{x_{k}}{\|x_{k}\|}\rightarrow v_0\in S_{\infty}\backslash\{0\}$. For a fixed $x_0\in S$, we have $x_0\in S\cap k\mathbf{B}$ for $k$ large enough. Since $x_k\in S=K_{\bar x}$ for all $k$, we have $f_i(x_k)\leq f_i(\bar x)$ for each $i\in\{1,2,\dots,q\}$. Dividing the both sides of these inequalities by $\|x_k\|^{d_i}$  and letting $k\rightarrow +\infty$, we get that
\begin{equation}\label{241226c}
		(f_i)^\infty_{d_i}(v_0)\leq 0, i\in\{1,2,\dots,q\}.
\end{equation}
By condition \emph{(i)}, we have $SOL^{w}(S_\infty, f^\infty_\mathbf{d})\neq\emptyset$. It follows from Proposition \ref{necessary} that $0\in SOL^{w}(S_\infty, f^\infty_\mathbf{d})$. If  $v_0\notin SOL^{w}(S_\infty, f^\infty_\mathbf{d})\setminus\{0\}$, then there exists $v'\in S_\infty$ such that $(f_{i})^\infty_{d_{i}}(v')<(f_{i})^\infty_{d_{i}}(v_0)$ for all $\{1,2,\dots,q\}$. So $0\notin SOL^{w}(S_\infty, f^\infty_\mathbf{d})$ by inequalities (\ref{241226c}), which is a contradiction with $0\in SOL^{w}(S_\infty, f^\infty_\mathbf{d})$. Thus, we have $v_0\in SOL^{W}(S_\infty, f^\infty_\mathbf{d})\setminus\{0\}$. By condition \emph{(ii)}, we have that there exists $t_0>0$ such that $f(x-t_0v_0)\leq f(x)$ for all $x\in S$. And it follows from $\sum_{i=1}^{q}\lambda_i f_i(x_k)\leq \sum_{i=1}^{q}\lambda_i f_i(x)$ for all $x\in S\cap k\mathbf{B}$ (since $x_k\in SOL(S\cap k\mathbf{B}, \sum^{s}_{i=1}\lambda_i f_i)$ for all sufficiently large $k$) that $\sum_{i=1}^{q}\lambda_i f_i(x_k-t_0v_0)\leq \sum_{i=1}^{q}\lambda_i f_i(x)$ for all $x\in S\cap k\mathbf{B}$.
Since $\{x_k\}\subseteq S\cap k\mathbf{B}\subseteq S$ and $f(x_k-t_0v_0)\leq f(x_k)$ for all $k$, we have $f(x_k-t_0v_0)\leq f(\bar x)$ for all $k$. It follows from $\{x_k-t_0v_0\}\subseteq K$ that $\{x_k-t_0v_0\}\subseteq S$.
For all $k$ large enough such that $0<\frac{t_0}{\|x_k\|}<1$ and $\|\frac{x_k}{\|x_k\|}-v_0\|<1$, we have
\begin{equation*} 
\begin{split}
\|x_k-t_0v_0\|&=\|(1-\frac{t_0}{\|x_k\|})x_k+t_0(\frac{x_k}{\|x_k\|}-v_0)\|\\
&\leq(1-\frac{t_0}{\|x_k\|})\|x_k\|+t_0\|\frac{x_k}{\|x_k\|}-v_0\|\\
&=\|x_k\|+t_0(\|\frac{x_k}{\|x_k\|}-v_0\|-1).
\end{split}
\end{equation*}
Thus, $\|x_k-t_0v_0\|<\|x_k\|\leq k$. So $x_k-t_0v_0\in S\cap k\mathbf{B}$ for all $k$ large enough. Therefore, $x_k-t_0v_0\in SOL(S\cap k\mathbf{B}, \sum^{q}_{i=1}\lambda_i f_i)$, which is a  contradiction with (\ref{241226b}). So the sequence $\{x_k\}$ is bounded. Without loss of generality, we assume that $\|x_{k}\|\to x_0$ as $k\to +\infty$. We claim that $x_0\in SOL(S, \sum^{q}_{i=1}\lambda_i f_i)$. If not, then there exists $x_1\in S$ such that $\sum^{q}_{i=1}\lambda_i f_i(x_1)<\sum^{q}_{i=1}\lambda_i f_i(x_0)$. Then for all $k$ large enough, we have $x_1\in S\cap k\mathbf{B}$ and $\sum^{q}_{i=1}\lambda_i f_i(x_1)<\sum^{q}_{i=1}\lambda_i f_i(x_k)$, which is a  contradiction with $x_k\in SOL(S\cap k\mathbf{B}, \sum^{q}_{i=1}\lambda_i f_i)$ for all sufficiently large $k$. So $x_0\in SOL(S, \sum^{q}_{i=1}\lambda_i f_i)$. By \cite[Proposition 3.2]{24LDY3}, we deduce $SOL(S, \sum^{q}_{i=1}\lambda_i f_i)\subseteq SOL^{s}(S, f)\subseteq SOL^{s}(K, f)$. Thus, $SOL^{s}(K, f)$ is nonempty.\end{proof}

For any $x\in K$, the nonempty closed set $S\subseteq\mathbf{R}^n$ satisfying with  $(K_{\bar x})_\infty\subseteq S_\infty\subseteq K_\infty$ and $\lambda=(\lambda_1,\lambda_2,\dots,\lambda_q)\in \mathbf{R}^q_+\backslash\{0\}$, we know that $SOL(S_\infty, \{\sum_{i=1}^{q}\lambda_i f_i\}^{\infty}_d)\subseteq SOL^{w}(S_\infty, f^{\infty}_\mathbf{d})$ and $SOL^{s}(S_\infty, f^{\infty}_\mathbf{d})\subseteq SOL^{w}(S_\infty, f^{\infty}_\mathbf{d})$. Thus, if $SOL(S_\infty, \{\sum_{i=1}^{q}\lambda_i f_i\}^{\infty}_d)$ and $SOL^{s}(S_\infty, f^{\infty}_\mathbf{d})$ are unbounded, then $SOL^{w}(S_\infty, f^{\infty}_\mathbf{d})$ is unbounded. So by Theorem \ref{241227c}, we have the following two results.

\begin{corollary}\label{241227a}
If the following conditions hold:
\begin{itemize}
	\item [(i)] For any $x\in K$, the nonempty closed set $S\subseteq\mathbf{R}^n$ satisfying with  $(K_{\bar x})_\infty\subseteq S_\infty\subseteq K_\infty$ and $\lambda=(\lambda_1,\lambda_2,\dots,\lambda_q)\in \mathbf{R}^s_+\backslash\{0\}$ with index set $I=\{i\in \{1,2,\dots,q\}\vert\lambda_i\neq 0\}$ such that $SOL(S_\infty, \{\sum_{i=1}^{q}\lambda_i f_i\}^{\infty}_d)$ with $d=\deg \sum_{i=1}^{q}\lambda_i f_i$ is unbounded, this is, the vector polynomial $f$ is relatively $I$-$\mathbf{R}^q_+$-zero-non-regular on $K$.
	\item [(ii)] And for every $v\in SOL(S_\infty, \{\sum_{i=1}^{q}\lambda_i f_i\}^{\infty}_d)\setminus\{0\}$, there exists $t>0$ such that $x-tv\in K$ and  $f(x-tv)\leq f(x)$ for all $x\in S$.
\end{itemize}
Then $SOL^{s}(K, f)$ is nonempty.
\end{corollary}

\begin{corollary}\label{241227b}
If the following conditions hold:
\begin{itemize}
	\item [(i)] For any $x\in K$ and  nonempty closed set $S\subseteq\mathbf{R}^n$ satisfying with  $(K_{\bar x})_\infty\subseteq S_\infty\subseteq K_\infty$ such that the set $SOL^{s}(S_\infty, f^{\infty}_\mathbf{d})$ is unbounded, this is, the vector polynomial $f$ is relatively weakly non-regular on $K$.
	\item [(ii)] And for every $v\in SOL^{s}(S_\infty, f^{\infty}_\mathbf{d})\setminus\{0\}$, there exists $t>0$ such that $x-tv\in K$ and $f(x-tv)\leq f(x)$ for all  $x\in S$.
\end{itemize}
Then $SOL^{s}(K, f)$ is nonempty.
\end{corollary}

Now, we give a following example to illustrate Theorem \ref{241227c}.
\begin{example}
	Consider the polynomial $f=(f_{1}, f_{2})$ with
$$f_{1}(x_{1}, x_{2})=x^3_{1},\quad f_{2}(x_{1}, x_{2})=x_{1}$$ and
$$K=\{(x_{1}, x_{2})\in \mathbf{R}^{2}: x_{1}\geq 0,e^{x_{1}}-x_{1}\geq 0\}.$$
	Then $(f_{1})^{\infty}_{d_1}(x_{1}, x_{2})=x^3_{1}, (f_{2})^{\infty}_{d_2}(x_{1}, x_{2})=x_{1}$. Let $\bar x\in K$ and the set $S\subseteq\mathbf{R}^n$ satisfying with  $(K_{\bar x})_\infty\subseteq S_\infty\subseteq K_\infty$ be arbetrary. Then we have $(K_{\bar x})_\infty=\{(x_{1}, x_{2})\in \mathbf{R}^{2}: x_{1}=0, x_2\in\mathbf{R}\}\subseteq S_\infty\subseteq K_\infty=\{(x_{1}, x_{2})\in \mathbf{R}^{2}: x_{1}\geq 0, x_2\in\mathbf{R}\}$. We can calculate $SOL^{w}(S_\infty, f^{\infty}_\mathbf{d})=\{(x_{1}, x_{2})\in \mathbf{R}^{2}: x_{1}=0, x_2\in\mathbf{R}\}$, which is unbounded. Let $v=(0, v_2)$ with $v_2\in \mathbf{R}$. Then we can easy to prove $x-tv\in K$ and $f(x-tv)\leq f(x)$ for all $x\in S$ and all $t>0$ small enough. By Theorem \ref{241227c}, we know that $SOL^{s}(K, f)$ is nonempty. Clearly, $\{(x_1, x_2)\in \mathbf{R}^2 \vert x_1=0, x_2\in \mathbf{R} \}\subseteq SOL^{s}(K, f)$, which is also unbounded.
\end{example}

\section{Local properties and genericity of relative regularity conditions}

In this section, we investigate  local properties and genericities of relative $\mathbf{R}^q_+$-zero-regularity, relatively weak regularity and relatively strong regularity of ${\rm{PVOP}}(K, f)$. Given an integer $d$, in what follows, we always  let $\mathbf{P}_{d}$ denote the family of all polynomials of degree at most $d$, and
	$$X^{n}_{d}(x)=(1, x_{1}, \dots, x_{n}, x^{2}_{1}, \dots, x^{2}_{n}, \dots, x^{d}_{1}, x^{d-1}_{1}x_{2}, x^{d-1}_{1}x_{3}, \dots, x^{d}_{2}, x^{d-1}_{2}x_{3}, x^{d-1}_{2}x_{3}, \dots, x^{d}_{n}),$$
whose components are listed by  the lexicographic ordering. The dimension of $\mathbf{P}_{d}$ is  denoted by $\kappa_d$. Then, for each polynomial $p\in \mathbf{P} _{d}$,  there exists a unique $\alpha\in \mathbf{R}^{\kappa_d}$ such that $p(x)=\langle \alpha, X^{n}_{d}(x)\rangle$. $\mathbf{P} _{d}$  can be endowed with a norm $\|p\|=\|\alpha\|=\sqrt{\alpha^{2}_{1}+\dots+\alpha^{2}_{\kappa_d}}$. Let $p^k\in \mathbf{P} _{d}$ with $p^k\to p\in \mathbf{P} _{d}$ and $x^k\in \mathbf{R}^n$ with $x^k\to x\in \mathbf{R}^n$. It is easy to verify that $(p^k)^{\infty}\to p^{\infty}$ and $p^k(x^k)\to p(x)$ as $k\to +\infty$.

Given  $\mathbf{d}=(d_{1},  \dots, d_{q})\in \mathbf{R}^{q}$ with $d_i$ being an integer, $i=1,\cdots, q$,  let $\mathbf{P}_{\mathbf d}=\mathbf{P}_{d_{1}}\times\dots\times\mathbf{P}_{d_{s}}$. Denoted by $\mathbf{GR}^{\mathbf d}$ the family of  all vector polynomials $p$  with $\deg p_{i}=d_{i}, i=1, \dots, q$, such that for some set $S$ and nonempty index set $I\subseteq\{1,2,\dots,q\}$, $p$ is relatively $I$-$\mathbf{R}^q_+$-zero-regular with $S$ on $K$ and $\mathbf{GR}^{\mathbf d}_{w}$  (resp. $\mathbf{GR}^{\mathbf d}_{s}$) the family of  all vector polynomials $p$  with $\deg p_{i}=d_{i}, i=1, \dots, q$, such that for some set $S$, $p$ is relatively strongly (resp. weakly) regular with $S$ on $K$.

\subsection{Local properties of relative regularity conditions}
In this subsection, we discuss the local properties of the relative regularity conditions of polynomial optimization problems.

\begin{proposition}
	$\mathbf{GR}^{\mathbf{d}}$, $\mathbf{GR}^{\mathbf{d}}_{s}$ and $\mathbf{GR}^{\mathbf{d}}_{w}$ are nonempty.
\end{proposition}
\begin{proof}
	We only need to prove that there exists $\lambda=(\lambda_1,\lambda_2,\dots,\lambda_q)\in \rm{ int }\mathbf{R}^q_+$ such that $\{\sum_{i=1}^{q}\lambda_i f_i\}^{\infty}_d\in \mathbf{GR}^{\mathbf{d}}$, since if hold, then $\{\sum_{i=1}^{q}\lambda_i f_i\}^{\infty}_d\in\mathbf{GR}^{\mathbf{d}}_{s}\subseteq \mathbf{GR}^{\mathbf{d}}_{w}$. Let $\lambda=(\lambda_1,\lambda_2,\dots,\lambda_q)\in \rm{ int }\mathbf{R}^q_+$. If $K$ is bounded, then for any $S_\infty\subseteq K_\infty$, $S_{\infty}=\{0\}$. In this case ${SOL}(S_{\infty}, \{\sum_{i=1}^{q}\lambda_i f_i\}^{\infty}_d)=\{0\}$, and so $f\in \mathbf{GR}^{\mathbf{d}}$. Suppose that $K$ is unbounded. Let $S=K$. Then there exists $x^{*}=(x^{*}_{1},  \dots, x^{*}_{n})\in S_{\infty}\backslash \{0\}$. Without loss of generality, we suppose that  $x^{*}_{i_{0}}\neq 0$. Consider the vector polynomial $f=(f_{1},  \dots, f_{q}): \mathbf{R}^{n} \mapsto \mathbf{R}^{q}$ with $f_{i}(x)=-(x^{*}_{i_{0}}x_{i_{0}})^{d_{i}},  i=1,\cdots, q$. Then $f_{i}(x)$ is a polynomials $f$ of degree $d_i$ and  $f_{i}(tx^{*})=-(x^{*}_{i_{0}})^{2d_{i}}t^{d_{i}}\to -\infty$ as $t\to +\infty$.  As a consequence,  ${SOL}^{w}(S_{\infty}, \{\sum_{i=1}^{q}\lambda_i f_i\}^{\infty}_d)=\emptyset$, and so $f\in \mathbf{GR}^{\mathbf{d}}$.\end{proof}

\begin{proposition}\label{open}
	 $\mathbf{GR}^{\mathbf{d}}$ and $\mathbf{GR}^{\mathbf{d}}_{s}$ are open in $\mathbf{P}_{\mathbf{d}}$.
\end{proposition}
\begin{proof}
	We shall prove that $\mathbf{P}_{\mathbf{d}}\backslash\mathbf{GR}^{\mathbf{d}}$ is  closed in $\mathbf{P}_{\mathbf{d}}$. Let $\{f^{k}\}\subseteq \mathbf{P}_{\mathbf{d}}\backslash\mathbf{GR}^{\mathbf{d}}_{s}$ with $f^{k}=(f^{k}_{1},  \dots, f^{k}_{q})$ such that $f^{k}=(f^{k}_{1},  \dots, f^{k}_{q})\rightarrow f=(f_{1},  \dots, f_{q})$ as $k\to +\infty$.  We suppose that $\mbox {deg }f_{i}=d_{i}$ for all $i\in \{1, 2, \dots, q\}$ since  $f\notin \mathbf{GR}^{\mathbf{d}}$ when  $\mbox{ deg }f_{i_{0}}<d_{i_{0}}$ for some  $i_{0}\in \{1, \dots, q\}$, where $d_i$  is the $i$-th component of $\mathbf{d}$. Thus, we have $\deg f^k_i=d_i$ for all sufficiently large $k$ and all $i\in \{1,2,\dots,q\}$. Without loss of generality, we assume $\deg f^k_i=d_i$ for all $k$ and all $i\in \{1,2,\dots,q\}$. Let $\bar x\in K$, the nonempty closed set $S\subseteq\mathbf{R}^n$ satisfying with $(K_{\bar x})_\infty\subseteq S_\infty \subseteq K_\infty$ and $\lambda=(\lambda_1,\lambda_2,\dots,\lambda_q)\in \mathbf{R}^q_+\backslash\{0\}$ be arbitrary.
Since $SOL(S_\infty, \{\sum_{i=1}^{q}\lambda_i f^k_i\}^{\infty}_d)$ with $d=\deg \sum_{i=1}^{q}\lambda_i f^k_i$ is unbounded for all $k$, there exists  $x_k\in SOL(S_\infty, \{\sum_{i=1}^{q}\lambda_i f^k_i\}^{\infty}_d)$ such that $\|x_k\|\to +\infty$.  Without loss of generality, we assume that  $\frac{x_{k}}{\|x_{k}\|}\to x^{*}\in S_{\infty}\backslash \{0\}$.  We claim  $x^{*}\in SOL(S_\infty, \{\sum_{i=1}^{q}\lambda_i f_i\}^{\infty}_d)$. Indeed, if not, then there exists $v\in S_{\infty}$ such that
\begin{equation}\label{fypa1}
 \{\sum_{i=1}^{q}\lambda_i f_i\}^{\infty}_d(v)<\{\sum_{i=1}^{q}\lambda_i f_i\}^{\infty}_d(x^{*}).
\end{equation}
Since  $x_k\in SOL(S_\infty, \{\sum_{i=1}^{q}\lambda_i f^k_i\}^{\infty}_d)$ and $\|x_{k}\|v\in S_\infty$, we have
$$\{\sum_{i=1}^{q}\lambda_i f^k_i\}^{\infty}_d(\|x_{k}\|v)-\{\sum_{i=1}^{q}\lambda_i f^k_i\}^{\infty}_d(x_{k})\geq 0.$$
Since   $(f^{k}_{i})_{d_{i}}^{\infty}\to (f_i)^{\infty}_{d_i}$ as $k\to +\infty$, dividing the both sides of the above inequality by $\|x_{k}\|^{d}$ and then letting $k\to +\infty$, we get
$$\{\sum_{i=1}^{q}\lambda_i f_i\}^{\infty}_d(v)-\{\sum_{i=1}^{q}\lambda_i f_i\}^{\infty}_d(x^{*})\geq 0.$$
This reaches a contradiction to $(\ref{fypa1})$. So $x^{*}\in SOL(S_\infty, \{\sum_{i=1}^{q}\lambda_i f_i\}^{\infty}_d)\backslash\{0\}$, and so, $SOL(S_\infty, \{\sum_{i=1}^{q}\lambda_i f_i\}^{\infty}_d)$ is unbounded. By the arbitrariness of $\bar x$, $S$ and $\lambda$. we have $f\in\mathbf{P}_{\mathbf{d}}\backslash\mathbf{GR}^{\mathbf{d}}$. Thus, $\mathbf{P}_{\mathbf{d}}\backslash\mathbf{GR}^{\mathbf{d}}$ is closed.

As similar to \cite[Proposition 6.2]{LDY1}, we can also prove $\mathbf{GR}^{\mathbf{d}}_{s}$ is open in $\mathbf{P}_{\mathbf{d}}$.
\end{proof}

\begin{remark}\label{questions}
	When $q=1$, Proposition \ref{open} reduces to  \cite[Lemma 4.1]{HV1}. The following example shows that $\mathbf{GR}^{\mathbf{d}}_{w}$ may not be open in $\mathbf{P}_{\mathbf{d}}$.
\end{remark}

\begin{example}\label{241225a}
Consider the polynomial $f=(f_{1}, f_{2})$ with
$$f_{1}(x_{1}, x_{2})=x_{1},\quad f_{2}(x_{1}, x_{2})=x_{2}$$ and
$$K=\{(x_{1}, x_{2})\in \mathbf{R}^{2}: x_{1}\geq 0\}.$$
	Then $(f_{1})^{\infty}_{d_1}(x_{1}, x_{2})=x_{1}, (f_{2})^{\infty}_{d_2}(x_{1}, x_{2})=x_{2}$. On the one hand, let $\bar x=(0, 0)$ and $S=K_{\bar x}$. Then $SOL(S_\infty, f^{\infty}_\mathbf{d})=\emptyset$. Thus, we have $f\in\mathbf{GR}^{\mathbf{d}}_{w}$. On the other hand, let $f^n=(f^n_{1}, f^n_{2})$ with $f^n_{1}=x_2, f^n_{2}=x_1-\frac{1}{n}x_2$ and $x\in K$, and let the set $S\subseteq\mathbf{R}^n$ satisfy with $(K_{\bar x})_\infty\subseteq S_\infty\subseteq K_\infty$ be arbitrary. Obviously, $S_\infty$ is unbounded and $f^n\to f$ as $n\to+\infty$. However, it is easy to prove $SOL(S_\infty, (f^n)^{\infty}_\mathbf{d})=S_\infty\bigcap\{(x_{1}, x_{2})\in \mathbf{R}^{2}: x_{1}=0\}$, which is unbounded. Thus, $f^n\notin\mathbf{GR}^{\mathbf{d}}_{w}$. So $\mathbf{GR}^{\mathbf{d}}_{w}$ is not open in $\mathbf{P}_{\mathbf{d}}$.
\end{example}

In following result, we shall show that relative $I$-$\mathbf{R}^q_+$-zero-regularity of a vector polynomial remains stable under a small perturbation.

\begin{theorem}\label{local R+regu}
Let $\bar x\in K$, $\lambda=(\lambda_1,\lambda_2,\dots,\lambda_q)\in \mathbf{R}^q_+\backslash\{0\}$ and $S\subseteq\mathbf{R}^n$ satisfy with $(K_{\bar x})_\infty\subseteq S_\infty\subseteq K_\infty$. Then the following conclusions hold:
	\item{(i)} If $SOL(S_\infty, \{\sum_{i=1}^{q}\lambda_i f_i\}^{\infty}_d)=\{0\}$, then there exists $\epsilon>0$ such that $SOL(S_\infty, \{\sum_{i=1}^{q}\lambda_i f_i\}^{\infty}_d)=\{0\}$ for all $g\in \mathbf{P}_{\mathbf{d}}$ satisfying $\|g-f\|<\epsilon$;
	\item{(ii)} If $SOL(S_\infty, \{\sum_{i=1}^{q}\lambda_i f_i\}^{\infty}_d)=\emptyset$, then there exists $\epsilon>0$ such that $SOL(S_\infty, \{\sum_{i=1}^{q}\lambda_i f_i\}^{\infty}_d)=\emptyset$ for all $g\in \mathbf{P}_{\mathbf{d}}$ satisfying $\|g-f\|<\epsilon$.
\end{theorem}
\begin{proof}
Since $\mathbf{GR}^{\mathbf{d}}$ is open in $\mathbf{P}_{\mathbf{d}}$  (by Proposition \ref{open}) and $f\in\mathbf{GR}^{\mathbf{d}}$, there exists an open ball $\mathbf{B}(f, \delta)\subseteq\mathbf{GR}^{\mathbf{d}}$ such that either $SOL(S_\infty, \{\sum_{i=1}^{q}\lambda_i g_i\}^{\infty}_d)=\{0\}$ or $SOL(S_\infty, \{\sum_{i=1}^{q}\lambda_i g_i\}^{\infty}_d)=\emptyset$ for all $g=(g_1,g_2,\dots,g_q)\in \mathbf{B}(f, \delta)$. Since $g\in \mathbf{B}(f, \delta)$, we can suppose $\deg g=\mathbf{d}$ for all $g\in \mathbf{B}(f, \delta)$. Let $d=\deg \sum_{i=1}^{q}\lambda_i g_i$.

\emph{(i)} It suffices to show that there exists $\epsilon\in (0,\delta)$ such that $SOL(S_\infty, \{\sum_{i=1}^{q}\lambda_i g_i\}^{\infty}_d)=\{0\}$ for all $g=(g_1,g_2,\dots,g_q)\in \mathbf{B}(f, \epsilon)$ when $SOL(S_\infty, \{\sum_{i=1}^{q}\lambda_i f_i\}^{\infty}_d)=\{0\}$. Suppose on the contrary that  for any $\epsilon\in (0,\delta)$, there exists $g^{\epsilon}=(g^\epsilon_1,g^\epsilon_2,\dots,g^\epsilon_q)\in \mathbf{P}_{\mathbf{d}}$ with $\|g^{\epsilon}-f\|<\epsilon$  such that $SOL(S_{\infty}, \{\sum_{i=1}^{q}\lambda_i g^{\epsilon}_i\}^{\infty}_d)=\emptyset$.
It follows that there exists $x_{\epsilon}\in S_{\infty}\backslash\{0\}$ such that
 	\begin{equation}\label{3.425}
 		(\sum_{i=1}^{q}\lambda_i g_i^{\epsilon})_{d}^{\infty}(x_{\epsilon})<(\sum_{i=1}^{q}\lambda_i g_i^{\epsilon})_{d}^{\infty}(0)=0.
 	\end{equation}
Since $g^{\epsilon}\in \mathbf{B}(f, \epsilon)\subset  \mathbf{GR}^{\mathbf{d}}$, we get $\deg (\sum_{i=1}^{q}\lambda_i g_i^{\epsilon})=d$. Because $g^{\epsilon}\to f$ as $\epsilon\to 0$, we have $(\sum_{i=1}^{q}\lambda_i g_i^{\epsilon})_{d}^{\infty}\to (\sum_{i=1}^{q}\lambda_i f_i)_{d}^{\infty}$ as $\epsilon\to 0$. Without loss of generality, we assume that $\frac{x_{\epsilon}}{\|x_{\epsilon}\|}\to x^*\in S_{\infty}\backslash \{0\}$ as $\epsilon\to 0$.  Dividing the both sides of (\ref{3.425}) by $\|x_{\epsilon}\|^{d}$ and then letting $\epsilon\rightarrow 0$, we get
$$(\sum_{i=1}^{q}\lambda_i f_{i})^{\infty}_{d}(x^*)\leq 0.$$
It follows that $x^*\in SOL(S_\infty, \{\sum_{i=1}^{q}\lambda_i f_i\}^{\infty}_d)\setminus\{0\}$, which  reaches a contradiction to $SOL(S_\infty, \{\sum_{i=1}^{q}\lambda_i f_i\}^{\infty}_d)=\{0\}$.

\emph{(ii)}  It suffices to show that there exists $\epsilon\in (0,\delta)$ such that $SOL(S_\infty, \{\sum_{i=1}^{q}\lambda_i g_i\}^{\infty}_d)=\emptyset$ for all $g=(g_1,g_2,\dots,g_q)\in \mathbf{B}(f, \epsilon)$ when $SOL(S_\infty, \{\sum_{i=1}^{q}\lambda_i f_i\}^{\infty}_d)=\emptyset$. Suppose on the contrary that  for any $\epsilon\in (0,\delta)$, there exists $g^{\epsilon}=(g^\epsilon_1,g^\epsilon_2,\dots,g^\epsilon_q)\in \mathbf{P}_{\mathbf{d}}$ with $\|g^{\epsilon}-f\|<\epsilon$  such that $SOL(S_{\infty}, \{\sum_{i=1}^{q}\lambda_i g^{\epsilon}_i\}^{\infty}_d)=\{0\}$. It follows that
 	\begin{equation*}\label{3.5}
 		0=(\sum_{i=1}^{q}\lambda_i g_i^{\epsilon})^{\infty}_{d}(0)\leq (\sum_{i=1}^{q}\lambda_i g_i^{\epsilon})^{\infty}_{d}(v)
 	\end{equation*}
for any $v\in S_{\infty}$. Since $g^{\epsilon}\in \mathbf{B}(f, \epsilon)\subset  \mathbf{GR}^{\mathbf{d}}$, we get $\deg (\sum_{i=1}^{q}\lambda_i g_i^{\epsilon})=d$. Because $g^{\epsilon}\to f$ as $\epsilon\to 0$, we have $(\sum_{i=1}^{q}\lambda_i g_i^{\epsilon})_{d}^{\infty}\to (\sum_{i=1}^{q}\lambda_i f_i)_{d}^{\infty}$ as $\epsilon\to 0$. Letting $\epsilon\rightarrow 0$ in the above inequality, we get
$$0=(\sum_{i=1}^{q}\lambda_i f_{i})^{\infty}_{d}(0)\leq (\sum_{i=1}^{q}\lambda_i f_{i})^{\infty}_{d}(v).$$
Since $v\in S_{\infty}$ is arbitrary, we get $0\in SOL(S_\infty, \{\sum_{i=1}^{q}\lambda_i f_i\}^{\infty}_d)$, a contradiction.
\end{proof}

Similar to the proof of Theorem 4.4 in \cite{LDY1}, we can also obtain that the relatively strong regularity of a vector polynomial remains stable under a small perturbation as follows and we omit its proof.

\begin{theorem}\label{local regu}
Let $\bar x\in K$ and $S\subseteq\mathbf{R}^n$ satisfy with $(K_{\bar x})_\infty\subseteq S_\infty\subseteq K_\infty$. Then the following conclusions hold:
	\item{(i)} If ${SOL}^{w}(S_{\infty}, f^{\infty}_{\mathbf{d}})=\{0\}$, then there exists $\epsilon>0$ such that ${SOL}^{w}(S_{\infty}, f^{\infty}_{\mathbf{d}})=\{0\}$ for all $g\in \mathbf{P}_{\mathbf{d}}$ satisfying $\|g-f\|<\epsilon$;
	\item{(ii)} If $SOL^{w}(S_{\infty}, f^{\infty}_{\mathbf{d}})=\emptyset$, then there exists $\epsilon>0$ such that ${SOL}^{w}(S_{\infty}, f^{\infty}_{\mathbf{d}})=\emptyset$ for all $g\in \mathbf{P}_{\mathbf{d}}$ satisfying $\|g-f\|<\epsilon$.
\end{theorem}

\begin{remark}
By Example \ref{241225a}, we see that the relatively weak regularity of a vector polynomial dose not have stability result under a small perturbation as Theorems \ref{local R+regu} and \ref{local regu}.
\end{remark}

Observe that $f^{\infty}_{\mathbf{d}}=(f+g)^{\infty}_{\mathbf{d}}$ for all $g=(g_{1},  \dots, g_{q})\in \mathbf{P}_{\mathbf{d}}$ with $\deg g_{i}< \deg f_{i}$, $i=1, \dots, q$. As a consequence, we have the following result.

\begin{proposition}\label{perturb} Let $\bar x\in K$, $\lambda=(\lambda_1,\lambda_2,\dots,\lambda_q)\in \mathbf{R}^q_+\backslash\{0\}$ and $S\subseteq\mathbf{R}^n$ satisfy with $(K_{\bar x})_\infty\subseteq S_\infty\subseteq K_\infty$. Then for any vector polynomial $g=(g_{1}, \dots, g_{q})$ with $\deg g_{i}<\deg  f_{i}, i=1, \cdots, q$, the following conclusions hold:	
\begin{itemize}
\item[(i)] If $SOL(S_\infty, \{\sum_{i=1}^{q}\lambda_i f_i\}^{\infty}_d)=\{0\}$, then $SOL(S_\infty, \{\sum_{i=1}^{q}\lambda_i (f_i+g_i)\}^{\infty}_d)=\{0\}$.

\item[(ii)] If $SOL(S_\infty, \{\sum_{i=1}^{q}\lambda_i f_i\}^{\infty}_d)=\emptyset$, then $SOL(S_\infty, \{\sum_{i=1}^{q}\lambda_i (f_i+g_i)\}^{\infty}_d)=\emptyset$.

\item[(iii)]  If ${SOL}^{s}(K_{\infty}, f^{\infty}_{\mathbf{d}})=\{0\}$, then ${SOL}^{s}(K_{\infty}, (f+g)^{\infty}_{\mathbf{d}})= \{0\}$.

\item[(iv)]  If ${SOL}^{s}(K_{\infty}, f^{\infty}_{\mathbf{d}})=\emptyset$, then ${SOL}^{s}(K_{\infty}, (f+g)^{\infty}_{\mathbf{d}})= \emptyset$.

\item[(v)] If ${SOL}^{w}(K_{\infty}, f^{\infty}_{\mathbf{d}})=\{0\}$, then ${SOL}^{w}(K_{\infty}, (f+g)^{\infty}_{\mathbf{d}})= \{0\}$.

\item[(vi)] If ${SOL}^{w}(K_{\infty}, f^{\infty}_{\mathbf{d}})=\emptyset$, then ${SOL}^{w}(K_{\infty}, (f+g)^{\infty}_{\mathbf{d}})= \emptyset$.
\end{itemize}
\end{proposition}

The following result is a direct consequence of Theorems \ref{local R+regu} and \ref{local regu}, and Proposition \ref{perturb}.

\begin{corollary}
Let the nonempty index set $I\subseteq\{1,2,\dots,q\}$. For any vector polynomial $g=(g_{1}, \dots, g_{q})$ with $\mathrm{ deg }\  g_{i}<\mathrm{ deg }\ f_{i}, i=1,\cdots, q$, the following conclusions hold:	
\begin{itemize}
\item[(i)] If $f$ is  $I$-relatively $\mathbf{R}^q_+$-zero-regular,  then $f+g$ is relatively $I$-$\mathbf{R}^q_+$-zero-regular.
\item[(ii)]  If $f$ is relatively weakly regular, then $f+g$ is relatively weakly  regular.	
\item[(iii)] If $f$ is relatively strongly regular, then $f+g$ is relatively strongly  regular.
\end{itemize}
\end{corollary}

\subsection{Genericities of the relative regularity conditions}

In this subsection, we discuss the genericities of the relative regularity conditions of vector polynomials. We assume that the constraint $K$ is denoted as follows
\begin{equation}\label{250101a}
K=\{x\in \mathbf{R}^n\vert g_i(x)\leq 0, i\in \{1,2,\dots,m\}\},
\end{equation}
where $g_i, i\in \{1,2,\dots,m\}$ are convex polynomial. By Remark 5.1 in \cite{HV1}, we know that the recession cone of $K$ is a nonempty polyhedral cone, and there exists a matrix $A\in \mathbf{R}^{m\times n}$ such that
\begin{equation}\label{250101b}
K_\infty=\{x\in \mathbf{R}^n\vert Ax\leq 0\}.
\end{equation}
We recall the definition of genericity as follows.

\begin{definition}
We say a subset $S$ is generic in $\mathbf{R}^n$, if $S$ contains a countable intersection of dense and open sets in $\mathbf{R}^n$.
\end{definition}

Clearly, if $S_1$ is generic in $\mathbf{R}^n$ and $S_1\subseteq S_2$ then $S_2$ also is generic in $\mathbf{R}^n$. To discuss the genericity of the relative regularity conditions, we need the following result.

\begin{lemma}\cite[Theorem 5.1]{HV1}\label{250101c}
Assume that $K$ be represented by (\ref{250101a}) and the cone $K_\infty$ represented by (\ref{250101b}), where $A$ is full rank. Then the set $\mathbf{G}^d$ is generic in $\mathbf{P}_{d}$, where $\mathbf{G}^d$ the family of  all polynomials $p$  with $\deg p=d$ such that $p$ is regular on $K$.
\end{lemma}

Next, we obtain a genericity result of the relative $I$-$\mathbf{R}^q_+$-zero-regularity as follows.

\begin{theorem}
Assume that $K$ be represented by (\ref{250101a}) and the cone $K_\infty$ represented by (\ref{250101b}), where $A$ is full rank. Then the set $\mathbf{GR}^{\mathbf{d}}$ is generic in $\mathbf{P}_{\mathbf{d}}$.
\end{theorem}
\begin{proof}
Let $\mathbf{G}^{d_i}, i\in\{1,2,\dots,q\}$ be the family of  all polynomials $p$  with $\deg p=d_i$ such that ${\rm{ PSOP }}(K, p)$ is regular, $\mathbf{G}^\mathbf{d}=\mathbf{G}^{d_1}\times\mathbf{G}^{d_2}\times\cdots\times\mathbf{G}^{d_q}$, and let $h=(h_1,h_2,\dots,h_q)\in \mathbf{G}^\mathbf{d}$ be arbitrary. Then we only let $\lambda=(\lambda_1,\lambda_2,\dots,\lambda_q)\in \mathbf{R}^q_+\backslash\{0\}$ with $\lambda_{i_0}=1$ and $\lambda_{i}=0,i\in\{1,2,\dots,i_0-1,i_0+1,\dots,q\}$ and $S=K$. Then we have $SOL(S_\infty, \{\sum_{i=1}^{q}\lambda_i h_i\}^{\infty}_d)=SOL(K_\infty, (h_i)^{\infty}_{d_{i_0}})$. Since $h_{i_0}\in \mathbf{G}^{d_{i_0}}$, we have $SOL(S_\infty, \{\sum_{i=1}^{q}\lambda_i h_i\}^{\infty}_d)$ is bounded. Therefore, we have $h\in \mathbf{GR}^{\mathbf{d}}$. By the arbitrariness of $h\in \mathbf{G}^\mathbf{d}$, we can know $\mathbf{G}^\mathbf{d}\subseteq \mathbf{GR}^{\mathbf{d}}$. Thus, the set $\mathbf{GR}^{\mathbf{d}}$ is generic in $\mathbf{P}_{\mathbf{d}}$, since $\mathbf{G}^\mathbf{d}$ is generic in $\mathbf{P}_\mathbf{d}$ by Lemma \ref{250101c}.\end{proof}

However, the following example shows that the sets $\mathbf{GR}^{\mathbf d}_{w}$ and $\mathbf{GR}^{\mathbf d}_{s}$ may not be generic in $\mathbf{P}_{\mathbf{d}}$.

\begin{example}
Let $\mathbf{d}=(d_1,d_2)=(1,1)$ and $\mathbf{P}_{\mathbf{d}}=\mathbf{P}_{d_1}\times \mathbf{P}_{d_2}$, where
$$\mathbf{P}_{d_1}=\{a_2x_2+a_1x_1+a_0\vert (a_2,a_1,a_0)\in\mathbf{R}^3\}, \mathbf{P}_{d_2}=\{b_2x^2+b_1x+b_0\vert (b_2,b_1,b_0)\in\mathbf{R}^3\}.$$
Let $K=\{x=(x_1,x_2)\in \mathbf{R}^2\vert x_2\geq x_1\geq 0\}$. Then $K_\infty=K$. Consider the set
$$Q=\{(a_2x_2+a_1x_1+a_0,b_2x_2+b_1x_1+b_0)\vert a_1<0,a_2>0,b_1>0,b_2<0, a_2b_1-a_1b_2>0, a_0,b_0\in\mathbf{R}\}$$
Clearly, $Q$ is a open set in $\mathbf{P}_{\mathbf{d}}$. Let $h=(h_1,h_2)\in Q$. Then $(h_{1})^{\infty}_{d_1}(x_{1}, x_{2})=a_2x_2+a_1x_1, (f_{2})^{\infty}_{d_2}(x_{1}, x_{2})=b_2x_2+b_1x_1$.For any $\bar x\in K$, we can easy to calculate $(K_{\bar x})_\infty=\{x\in \mathbf{R}\vert x_2\geq x_1\geq 0\}$. Let the set $S$ with $(K_{\bar x})\subseteq S\subseteq K_\infty$ be arbitrary. Set $H=\{(x_1,x_2)\in\mathbf{R}^2\vert (a_2+b_2)x_2+(a_1+b_1)x_1=0\}$. Then $H$ is a unbounded set. Let $x\in H$ be arbitrary. Then we can easy to prove $x\in SOL^{s}(S_\infty, h^{\infty}_{\mathbf{d}})$, which implies $H\subseteq SOL^{s}(S_\infty, h^{\infty}_{\mathbf{d}})$, and so, $h\notin \mathbf{GR}^{\mathbf d}_{s}$. By the arbitrariness of $h\in Q$, we have $Q\cap\mathbf{GR}^{\mathbf d}_{w}=\emptyset$. Thus, $\mathbf{GR}^{\mathbf d}_{w}$ is not generic in $\mathbf{P}_{\mathbf{d}}$. Moreover, by $\mathbf{GR}^{\mathbf d}_{s}\subseteq\mathbf{GR}^{\mathbf d}_{w}$, we have that  $\mathbf{GR}^{\mathbf d}_{s}$ also is not generic in $\mathbf{P}_{\mathbf{d}}$.
\end{example}

\section{Conclusion}
In this paper, we extend and improve the concept of regularity conditions introduced by Hieu \cite{HV1} and Liu \cite{LDY1}, introducing the relative regularity conditions for polynomial vector optimization problem (see, Remark \ref{compa} \emph{(ii)}). We investigate the fundamental properties and characteristics of the relative regularity conditions. When the constraint is a closed set, we establish equivalence relationships between the $I$-Palais-Smale condition, weak $I$-Palais-Smale condition, $I$-M-tameness, $I$-properness and relative regularity conditions under the $I$-section-boundedness from below for some nonempty index $I\subseteq\{1,2,\dots,q\}$. Under the relative regularity conditions, we investigate nonemptiness of solution sets of a non-convex polynomial vector optimization problem on a nonempty closed set (not necessarily semi-algebraic set). As a consequence, we derive Frank-Wolfe type theorems for a non-convex polynomial vector optimization problem and provide a necessary and sufficient condition of existence of solution for a polynomial scalar optimization problem. Furthermore, even under the relative non-regularity conditions, we prove nonemptiness of solution sets of a non-convex polynomial vector optimization problem on a nonempty closed set. Finally, we explore local properties of relative $I$-$\mathbf{R}^q_+$-zero-regularity, relatively  weak regularity and strong regularity, along with their genericity under convex constraint set condition.  Our results extend and improve the corresponding results of \cite{DTN,Flores24,HV1,LGJ,LGJ1,LDY1,LDY2}.

\end{document}